\documentclass[reqno,12pt]{amsart}
\usepackage{amssymb,amsmath}
\usepackage[dvips]{color}
\usepackage[dvips,final]{graphics}
\usepackage{hyperref}

\usepackage{epstopdf}
\usepackage{amscd}
\usepackage[latin1]{inputenc}

\theoremstyle{plain}
\newtheorem{teo}{Theorem}[section]      
\newtheorem{prop}[teo]{Proposition}    
\newtheorem{lem}[teo]{Lemma}            

\theoremstyle{definition}               
\newtheorem{defin}[teo]{Definition}
\theoremstyle{remark}                   
\newtheorem{oss}[teo]{Remark}

\newcommand{\re}{\mathbb{R}}

\numberwithin{equation}{section}

\begin{document}

\title[Sharp energy estimates for nonlinear fractional diffusion equations]
{Sharp energy estimates for nonlinear fractional diffusion equations}

\thanks{Both authors were supported by grants
MTM2008-06349-C03-01 (Spain),  MTM2011-27739-C04-01 (Spain), and 2009SGR345 (Catalunya). 
The second author was partially supported by University of Bologna (Italy), funds for selected research topics.}

\author[Xavier Cabr\'e]{Xavier Cabr\'e}
\address{
ICREA and Departament de Matem\`atica Aplicada 1 \\
Universitat Polit\`ecnica de Catalunya, Diagonal 647 \\
08028 Barcelona (Spain) } \email{xavier.cabre@upc.edu}
\author[Eleonora Cinti]{Eleonora Cinti}
\address{Dipartimento di Matematica ``F. Casorati'', 
Universit\`a degli Studi di Pavia, Via Ferrata 1, 27100 Pavia (Italy)} \email{eleonora.cinti@unipv.it}

\begin{abstract} 
We study the nonlinear fractional equation
$(-\Delta)^su=f(u)$ in $\re^n,$
for all fractions $0<s<1$ and all nonlinearities $f$. For every fractional power $s\in (0,1)$,
we obtain sharp energy estimates for bounded global minimizers and for bounded monotone solutions.
They are sharp since they are
optimal for solutions depending only on one Euclidian variable.

As a consequence,
we deduce the one-dimensional symmetry of bounded global minimizers and of bounded
monotone solutions in dimension $n=3$ whenever $1/2\leq s<1$. This result is the analogue of a conjecture of De
Giorgi on one-dimensional symmetry for the classical equation
$-\Delta u=f(u)$ in $\re^n$. It remains open for $n=3$ and $s<1/2$, and also for $n\geq 4$ and all $s$.
\end{abstract}

\maketitle

\section{Introduction and results}

In this paper we establish energy estimates for some bounded solutions of
the fractional nonlinear equation
\begin{equation} \label{eq1-s}
(-\Delta)^{s} u=f(u)\quad \mbox{in}\;\re^n,\end{equation} for
every $0<s<1$, where $f:\re\rightarrow\re$ is a $C^{1,\gamma}$
function for some $\gamma>\max (0,1-2s)$.

In \cite{C-Cinti1}, we considered the case
$s=1/2$ and established sharp energy estimates for bounded global minimizers in
every dimension $n$, and for bounded monotone solutions in dimension
$n=3$. As a consequence, we deduced  one-dimensional (or 1-D)
symmetry for these types of solutions in dimension $n=3$.

This result about 1-D symmetry is the analogue of a conjecture of De
Giorgi for the Allen-Cahn equation $-\Delta u=u-u^3$ in $\re^n$.
More precisely, in 1978 De Giorgi conjectured that the level sets
of every bounded solution of the
Allen-Cahn equation which is monotone in one direction
must be hyperplanes, at least if
$n\leq 8$. That is, such solutions depend only on one Euclidean
variable. The conjecture was proven to be true in dimension
$n=2$ by Ghoussoub and Gui \cite{GG} and in dimension $n=3$ by
Ambrosio and the first author \cite{AC}. For $4\leq n\leq 8$, if
$\partial_{x_n}u>0$, and assuming the additional condition
$$\lim_{x_n \rightarrow \pm \infty}u(x',x_n)=\pm 1\quad \mbox{for all}\;x'\in \re^{n-1},$$ it has been
established by Savin \cite{S}. More recently, a counterexample to the
conjecture for $n\geq 9$ has been found by Del Pino, Kowalczyk,
and Wei \cite{dPKW}.

In this paper (see Theorem \ref{degiorgi-s} below), we establish the
one-dimensional symmetry of bounded global minimizers and of bounded monotone solutions of (\ref{eq1-s}) in
dimension $n=3$ for every $1/2\leq s<1$. This is the analogue of the
conjecture of De Giorgi for the operator $(-\Delta)^s$. Our result applies to all nonlinearities $f\in C^{1,\gamma}$, $\gamma > \max\{0,1-2s\}$.

In dimension $n=3$, in \cite{C-Cinti1} we proved the same result for
$s=1/2$. For $n=3$ and $0<s<1/2$ the question remains open,
as well as for  $n\geq 4$ and any $0<s<1$. For $n=2$ and $s=1/2$,
the one-dimensional symmetry of bounded stable solutions of (\ref{eq1-s})
was proven by the first author and Sol\`a-Morales
\cite{C-SM}. The same result in dimension $n=2$ for every
fractional power $0<s<1$ has been established by
the first author and Sire \cite{C-Si2} and by Sire and Valdinoci \cite{SV}. Recall that the class of stable solutions includes all global minimizers, as well as all monotone solutions.

The existence of 1-D monotone solutions (also called ``layers'') has been established by the first author and Sire in \cite{C-Si1, C-Si2} for all nonlinearities $f$ (not necessarily odd) for which its primitive (up to a sign) is a potential of ``double-well type''. In fact, \cite{C-Si1,C-Si2} establish that this is a necessary and sufficient condition on $f$ for a 1-D monotone solution to exist.
Our 1-D symmetry result for monotone solutions in $\re^3$ will use both the 1-D symmetry result in $\re^2$ of \cite{C-Si2} and (since it applies to all nonlinearities $f$) the necessary conditions on $f$ proved in \cite{C-Si1} for a monotone solution in $\re$ to exist.

As in \cite{C-Cinti1}, a crucial ingredient in the proof of 1-D
symmetry is a sharp energy estimate for bounded global minimizers and for bounded monotone solutions. 
These are Theorems \ref{energy-est-s} and \ref{energy-dim3-s}. Our estimates are sharp since they are
optimal for solutions depending only on one Euclidian variable.

The classical connection between the Allen-Cahn equation and minimal surfaces (or the perimeter
functional) is the motivation behind the conjecture of De Giorgi.
The following are some results in this direction but concerning the fractional Laplacian. 
In \cite{Mar} M.d.M.~Gonz\'alez proved that an energy functional related to fractional powers~$s$ of the 
Laplacian (for $1/2<s<1$) $\Gamma$-converges to the classical perimeter functional. 
The same result for $s=1/2$ had been proven by Alberti, Bouchitt\'e, and Seppecher in \cite{ABS}. 
In \cite{Ca-Sou} Caffarelli and Souganidis prove that scaled threshold dynamics-type algorithms 
corresponding to fractional Laplacians converge to certain moving fronts. More precisely, when $1/2\leq s<1$ the resulting 
interface moves by a weighted mean curvature, while for $0<s<1/2$ the normal velocity is nonlocal 
of fractional-type. Later, Caffarelli, Roquejoffre, and Savin \cite{CRS} set up the theory of nonlocal
$s$-minimal surfaces for $0<s<1/2$. Palatucci, Savin, and Valdinoci \cite{PSV,SV1,SV2} have established
precise relations between $s$-minimal surfaces and the interfaces of semilinear phase transition
equations driven by the fractional Laplacian. See Remark~\ref{minsurf} for more comments on
$s$-minimal surfaces.

To study the nonlocal problem (\ref{eq1-s}), we will realize it as
a local problem in $\re^{n+1}_+$ with a nonlinear Neumann
condition. More precisely, Caffarelli and Silvestre \cite{CS}
proved that $u$ is a solution of problem \eqref{eq1-s} in $\re^n$
if and only if $v$, defined on
$\re^{n+1}_+=\{(x,\lambda):x\in \re^n,\;\lambda>0\}$, is a
solution of the problem
\begin{equation}\label{eq2-s}
\begin{cases}
\mbox{div}(\lambda^{1-2s}\nabla v)=0& \text{in}\; \re_{+}^{n+1},\\
- d_{s}\lim_{\lambda\rightarrow 0}\lambda^{1-2s}\partial_{ \lambda}v=f(v)& \text{on}\; \re^{n},
\end{cases}
\end{equation} where $v(x,0)=u(x)$ on $\re^{n}=\partial \re_{+}^{n+1}$ and 
$$d_{s}=2^{2s-1}\frac{\Gamma(s)}{\Gamma(1-s)}$$ is a positive constant. 
The fact that this constant does not depend on $n$ is already shown in section 3.2 of \cite{CS}. Its precise value has been computed in several 
papers; see, e.g., \cite{FL,ST}. Using that $s\Gamma(s)=\Gamma(s+1)$ and $(1-s)\Gamma(1-s)=\Gamma(2-s)$ we deduce, respectively, that
$$\frac{d_s}{(2s)^{-1}}\rightarrow 1\:\:\mbox{as}\:\:s\rightarrow 0\quad \mbox{and}\quad \frac{d_s}{2(1-s)}\rightarrow 1\:\:\mbox{as}\:\:s\rightarrow 1.$$
Thus, $d_s$ blows-up as $s\rightarrow 0$ and $d_s$ tends linearly to zero as $s\rightarrow 1$.

In the sequel, the extension $v$ of $u$ in
$\re^{n+1}_+$ which satisfies $\mbox{div}(\lambda^{1-2s}\nabla v)=0$ will be named ``the $s$\textit{-extension} of $u$''.

Observe that for every $0<s<1$, we have that $-1<1-2s<1$ and thus
the weight $\lambda^{1-2s}$ which appears in \eqref{eq2-s} belongs
to the Muckenhoupt class $A_2$. As a consequence, the theory developed by Fabes,
Kenig, and Serapioni \cite{FKS} applies to problem \eqref{eq2-s} and thus a
Poincar\'e inequality, a Harnack inequality, and H\"{o}lder regularity
hold for solutions of our problem.
As shown by the first author and Sire \cite{C-Si1}, solutions to
\eqref{eq2-s} are $C^\beta(\overline{\re^{n+1}_+})$ for some $\beta\in (0,1)$.
But in general they have no further regularity in $\lambda$; note that $\lambda^{2s}$ (even that it is unbounded) solves
\eqref{eq2-s} with $f$ identically constant.
The trace  $u(\cdot,0)$ on $\{\lambda=0\}$ of a bounded solution of \eqref{eq2-s} is however a $C^{2,\beta}(\re^{n})$ function;
see \cite{C-Si1}. 
We will also use some gradient estimates from \cite{C-Si1} which apply to 
solutions of \eqref{eq2-s} (see Remark \ref{oss-grad} below) and  
that will be important in the proof of our energy estimates.

Problem (\ref{eq2-s}) associated to the nonlocal equation
\eqref{eq1-s} allows to introduce a notion of \textit{energy}
 and \textit{global minimality} for a solution $u$ of problem (\ref{eq1-s}).

Consider the cylinder
$$C_R=B_R \times (0,R)\subset \re^{n+1}_+,$$ where $B_R$ is the ball of radius $R$ centered at 0 in $\re^n$, 
and the energy functional
\begin{equation}\label{energia-s}
{\mathcal E}_{s,C_R}(v)=d_s\int_{C_R}\frac{1}{2}\lambda^{1-2s}|\nabla
v|^2dx d\lambda + \int_{B_R}G(v(x,0))dx, 
\end{equation} 
where $G'=-f$ and thus $G$ is defined up to an additive constant. Observe that ${\mathcal E}_{s,C_R}$ 
is the energy functional associated to problem \eqref{eq2-s}. 
We can now give the following definition. Let us denote, for $\Omega \subset \re^{n+1}_+$, the Sobolev space
$H^1(\Omega,\lambda^{1-2s})$ made of functions $v$ with $\lambda^{1-2s} (v^2+|\nabla v|^2)\in L^1(\Omega)$.

\begin{defin}\label{layer-s}
\begin{enumerate}
\item [a)] We say that a bounded $C^\beta_{\rm loc}(\overline{\re_+^{n+1}})\cap
H^1_{\rm loc}(\overline{\re_+^{n+1}},\lambda^{1-2s})$
function $v$ is a {\it global minimizer} of $(\ref{eq2-s})$ if, for
all $R>0$,
$${\mathcal E}_{s,C_R}(v)\leq {\mathcal E}_{s,C_R}(w),$$ for
every $H^1(C_R,\lambda^{1-2s})$ function $w$ such that
$v\equiv w$ in $\partial{C_R}\setminus\{\lambda=0\}$. 
\item [b)]
We say that a bounded $C^1$ function $u$ in $\re^{n}$ is a {\it
global minimizer} of $(\ref{eq1-s})$ if its $s$-extension $v$ is a
global minimizer of (\ref{eq2-s}). \item[c)] We say that a bounded
function $u$ is a \emph{layer solution} of
(\ref{eq1-s}) if $u$ is a solution, it is monotone increasing in one of the
$x$-variables, say $\partial_{x_n}u>0$ in $\re^n$, and
\begin{equation}\label{hp-limit-s}\lim_{x_n \rightarrow \pm
\infty}u(x',x_n)=\pm 1\quad \mbox{for every}\:x'\in
\re^{n-1}.\end{equation}
\end{enumerate}
\end{defin}

We recall that every layer solution is a global minimizer
(see \cite{C-Si2}). This is proved using the sliding method.
Note that the weight $\lambda^{1-2s}$ does not
depend on the horizontal variables $x_1,...,x_n$, and hence problem
\eqref{eq2-s} is invariant under translations in the directions
$x_1,...,x_n$.

Our first result is the following energy estimate for global
minimizers of problem (\ref{eq1-s}). In particular it applies also to layer solutions. Given a nonlinearity $f$ and a bounded function $u$
defined on $\re^n$, set
$$G(u)=\int_u^1 f\quad \mbox{and}$$
\begin{equation}\label{c_u-s}
c_u=\min\{G(s):\inf_{\re^n} u\leq s\leq \sup_{\re^n}
u\}.\end{equation}

\begin{teo}\label{energy-est-s}
Let $s\in (0,1)$, $f$ be any $C^{1,\gamma}$ nonlinearity with $\gamma>\max\{0,1-2s\}$, and let $u:\re^{n}\rightarrow \re$ be a bounded global minimizer of
(\ref{eq1-s}). Let $v$ be the $s$-extension of $u$ in
$\re^{n+1}_+$.

Then, for all $R>2$,
\begin{equation}\label{energy-s}
 \mathcal E_{s,C_R}(v)=d_s\int_{C_R}\frac{1}{2}\lambda^{1-2s}|\nabla v|^2dx
d\lambda+\int_{B_R}\{G(u)-c_u\}dx \leq C R^{n-2s}\int_{1/R}^1\rho^{-2s}d\rho,\end{equation}
where $c_u$ is defined by
(\ref{c_u-s}), and $C$ is a positive constant depending
only on $n$, $s$, $||f||_{C^{1,\gamma}([\inf_{\re^n}u,\sup_{\re^n}u])}$, and $||u||_{L^\infty(\re^n)}$.

As a consequence, for some constant $C$ depending on the same quantities as before, we have
\begin{equation}\label{energy-s1}
\mathcal E_{s,C_R}(v) \leq C R^{n-2s} \quad
\mbox{if}\;\;0<s<1/2,\end{equation}
\begin{equation}\label{energy-s2}
\mathcal E_{s,C_R}(v) \leq CR^{n-1}\log R \quad
\mbox{if}\;\;s=1/2,\end{equation}
\begin{equation}\label{energy-s3} \mathcal E_{s,C_R}(v) \leq C R^{n-1}\quad
\mbox{if}\;\;1/2<s<1.
\end{equation}
\end{teo}

For $s=1/2$, the estimate was proved in our previous paper \cite{C-Cinti1}.
For $s\in (0,1)$ they were announced in the second author Ph.D. Thesis \cite{Cinti}.
More recently the estimates 
have been proven with a different method (without using the
extension problem) by Savin and Valdinoci \cite{SV1}. 
While their proof is simpler than ours, we develope further results which are of independent interest. More precisely, an extension result (Theorem \ref{extension-s}) and a phase transition estimate (Theorem \ref{key-s}) both involving fractional Sobolev norms with weights.

\begin{oss}
The energy estimate \eqref{energy-s1} is sharp. Indeed, for every bounded solution of problem 
\eqref{eq1-s} the energy is also bounded below by $c_1R^{n-2s}$ for some constant $c_1>0$. 
This is a consequence of a monotonicity formula (Proposition \ref{mon-formula}) that we prove in section 3.

Moreover, also when $1/2\leq s<1$ the energy estimates \eqref{energy-s2} and \eqref{energy-s3} are 
sharp since they are optimal for 1-D layer solutions, in the sense that for these solutions the energy is also 
bounded below by $c_2R^{n-1}\log R$ (if $s=1/2$) and $c_2R^{n-1}$ (if $s>1/2$), for some constant $c_2>0$, when they are seen as solutions in $\re^n$. 
When $s>1/2$, this follows immediately from Fubini's theorem.
Indeed, if $v=v(x_n,\lambda)$ is a 1-D layer solution of \eqref{eq2-s} then, for $R>1$ we have (note that $G(u)-c_u\geq 0$ in $\re^n$)
\begin{eqnarray*}
{\mathcal E}_{s,C_{\sqrt{n}R}}(v) &\geq&   {\mathcal E}_{s,(-R,R)^n\times (0,R)}(v) =
d_s\int_{(-R,R)^n\times (0,R)}\frac{1}{2}\lambda^{1-2s}|\nabla v(x_n,\lambda)|^2 dx d\lambda + \\
&&\hspace{4.5cm} + \int_{(-R,R)^n}\{G(v(x_n,0))-c_u\}dx\\
&=&R^{n-1}\left\{d_s\int_0^R d\lambda \int_{-R}^R  \frac{1}{2}\lambda^{1-2s}|\nabla v|^2 dx_n  +
\int_{-R}^R\{G(v(x_n,0))-c_u\}dx_n\right\}\\
&\geq& R^{n-1}d_s\int_0^1 d\lambda \int_{-1}^1  \frac{1}{2}\lambda^{1-2s}|\nabla v(x_n,\lambda)|^2 dx_n
=CR^{n-1}.
\end{eqnarray*}
\end{oss}

In our next result we establish that in dimension $n=3$ the energy estimate \eqref{energy-s} holds also for bounded
monotone solutions without the limit assumption \eqref{hp-limit-s}.
These solutions can only be guaranteed to be minimizers among a certain class of functions (see Proposition \ref{monot-min-s}), but could fail to be global minimizers as defined before.

\begin{teo}\label{energy-dim3-s}
Let $n=3$, $f$ be any $C^{1,\gamma}$ nonlinearity with $\gamma
>\max\{0,1-2s\}$ and $u$ be a bounded solution of (\ref{eq1-s}) such that
$\partial_{e}u>0$ in $\re^3$ for some direction $e\in \re^3$,
$|e|=1$. Let $v$ be the $s$-extension of $u$ in $\re^{4}_+$. 

Then, $v$ satisfies the energy estimate \eqref{energy-s} with $n=3$.
\end{teo}

In dimension $n=3$ and for every $1/2\leq s<1$, Theorems \ref{energy-est-s}
and \ref{energy-dim3-s} lead to the 1-D symmetry of bounded global minimizers
and of bounded monotone solutions of problem (\ref{eq1-s}).
For $s=1/2$ this was proved in our previous paper \cite{C-Cinti1}.

\begin{teo}\label{degiorgi-s}
Assume that $n=3$ and $1/2\leq s<1$. Let $f$ be any $C^{1,\gamma}$
nonlinearity with $\gamma>\max\{0,1-2s\}$ and $u$ be either a
bounded global minimizer of \eqref{eq1-s}, or a bounded solution
of \eqref{eq1-s} which is monotone in some direction $e\in \re^3$, $|e|=1$, i.e., it satisfies $\partial_e u >0$ in $\re^3$.

Then, $u$ depends only on one variable, i.e., there exists $a \in
\re^3$ and $g: \re \rightarrow \re $, such that $u(x)=g(a\cdot x)$
for all $x \in \re^3$. Equivalently, the level sets of $u$ are
planes.\end{teo}

\begin{oss}\label{minsurf}
In \cite{CRS} Caffarelli, Roquejoffre, and Savin introduced and 
developed a regularity theory for nonlocal minimal surfaces. 
These surfaces, defined for $0<s<1/2$, 
can be interpreted as a non-infinitesimal version of classical minimal surfaces and arise when minimizing in an appropriate way the $H^s$-norm of the indicator function. Note
that when $0<s<1/2$ the indicator functions belong to $H^s$ and the extension problem 
\eqref{eq2-s} is a well posed problem for indicator functions. The authors also proved a sharp energy 
estimate $CR^{n-2s}$ related to ours in some sense: our equation is the Allen-Cahn approximation of 
these nonlocal minimal surfaces. The flatness of any $s$-minimal surface in all $\re^n$ is
only known in dimension $n=2$ by a recent result of Savin and Valdinoci~\cite{SV3} 
(this is the analogue statement to that of the conjecture of De Giorgi).
The same statement for $s$ sufficiently close to $1/2$ and $n\leq 7$
had been proved by Caffarelli and Valdinoci \cite{C-Val}.
\end{oss}

To prove 1-D symmetry, we use a Liouville type argument which requires an appropriate estimate
for the Dirichlet energy. By a result of Moschini \cite{mosch},
if the energy estimate
$$\int_{C_R}|\nabla v|^2dx d\lambda\leq CR^2 \log R$$ 
holds (note the exponent $2$ in $R^2\log R$), then one may use such Liouville type result and deduce 
1-D symmetry in $\re^n$ for global minimizers and for solutions which are
monotone in one direction.
Now, by Theorems \ref{energy-est-s} and
\ref{energy-dim3-s}, we have that
$$\int_{C_R}\frac{1}{2}\lambda^{1-2s}|\nabla v|^2dx d\lambda \leq
CR^2\log R
$$ 
holds for $n=3$ and every $1/2\leq s<1$.
Instead, if $0<s<1/2$, the sharp estimate is
$$\int_{C_R}\frac{1}{2}\lambda^{1-2s}|\nabla v|^2dx d\lambda \leq
CR^{3-2s}.$$ Since $3-2s>2$ when $0<s<1/2$, one can not use the Liouville argument.
This is the reason why we can prove 1-D symmetry only for $n\leq 3$, and in case $n=3$ only
for $1/2\leq s<1$.

We have two different proofs of our energy estimate \eqref{energy-s}.

The first one is simple but applies only to bistable nonlinearities (such as the Allen-Cahn nonlinearity 
$f(u)=u-u^3$) and to monotone solutions satisfying the limit assumption \eqref{hp-limit-s} or the more 
general \eqref{limit-s} below. We present this very simple proof in section~2. It was found by Ambrosio and 
the first author \cite{AC} to prove the optimal energy estimate for $-\Delta u=u-u^3$ in $\re^n$. 
In this specific case of bistable nonlinearities and monotone solutions satisfying \eqref{limit-s}, 
our energy estimate is uniform as $s\uparrow 1$. On the contrary, the energy estimate for global minimizers and general $f$ 
(as stated in Theorem \ref{energy-est-s}) is not uniform as $s \uparrow 1$.
Instead, the estimate in \cite{SV1} is uniform as $s \uparrow 1$, even if not stated in that paper.

Our second proof applies in more general situations and will lead to Theorems \ref{energy-est-s} 
and \ref{energy-dim3-s}. It is based on controlling the weighted $H^1(C_R,\lambda^{1-2s})$-norm of the solution
$v$ by certain weighted fractional Sobolev norms of the trace of $v$ on $\partial C_R$.

Let us recall the definition of the $H^{s}(A)$-norm of a function for
$0<s<1$, where $A$ is either a bounded Lipschitz domain of $\re^n$, or
$A=\partial \Omega$ and $\Omega$ is a bounded Lipschitz domain of
$\re^{n+1}$. It is given by
$$||w||^2_{H^{s}{(A)}}=||w||^2_{L^2(A)}+\int_{A}\int_{A}\frac{|w(z)-w(\overline{z})|^2}{|z-\overline{z}|^{n+2s}}d\mathcal{H}^n(z)
d\mathcal{H}^n(\overline{z}),$$ 
where $d\mathcal{H}^n$ denotes the $n$-dimensional Hausdorff measure. 
In the sequel we will use it for
$\Omega=C_1=B_1\times (0,1)\subset \re^{n+1}$ and $A=\partial C_1$.

To prove Theorem \ref{energy-est-s}, we use the following comparison argument. We construct a 
comparison function $\overline w$ which takes the same values of $v$ on $\partial C_R\cap \{\lambda>0\}$ and thus, 
by minimality of $v$,
$$\mathcal E_{C_R}(v)\leq  \mathcal E_{C_R}(\overline w).$$
Then, it is enough to estimate the energy of $\overline w$.

For simplicity, consider the case of the Allen-Cahn nonlinearity.
We define the function
$\overline{w}(x,\lambda)$ in $C_R$
in the following way. First we define $\overline{w}(x,0)$
on the base of the cylinder as a smooth function $g(x)$ which is
identically equal to $1$ in $B_{R-1}$ and $g(x)=v(x,0)$ for
$|x|=R$; then we define $\overline{w}(x,\lambda)$ as the unique
solution of the Dirichlet problem
\begin{equation}\begin{cases}
\text{div}(\lambda^{1-2s}\nabla \overline{w})=0 & \mbox{in} \;C_R\\
\overline{w}(x,0)=g(x)& \mbox{on}\;B_R \times\{\lambda=0\}\\
 \overline{w}(x,\lambda)=v(x,\lambda)& \mbox{on}\;\partial
 C_R\cap\{\lambda>0\}.\end{cases}\end{equation}
Since by definition $\overline w \equiv 1$ on $B_{R-1}\times \{0\}$, then the potential energy is bounded by 
$C|B_R\setminus B_{R-1}|=CR^{n-1}$. Thus, it remains to estimate the Dirichlet energy.

To do this we proceed in two steps. First, after rescaling, we apply Theorem \ref{extension-s} below, 
to control the Dirichlet norm of $\overline{w}_1$ (where $\overline{w}_1$ is the rescaled version of $\overline w$) in $C_1$ 
by some fractional Sobolev norms of its trace on $\partial C_1$. 
Then, we use Theorem \ref{key-s} below to give an estimate of these fractional norms.

We recall that in the proof of the estimate for the Dirichlet energy for $s=1/2$
a crucial point was an extension theorem which let us control
the $H^1(C_1)$-norm of a function by the $H^{1/2}(\partial C_1)$-norm of its trace. Here
we are in a more complicated situation, since we need to control the 
weighted $H^1(C_1,\lambda^{1-2s})$-norm, with a weight which degenerates or blows-up on a subset of $\partial C_1$
(its bottom).

We introduce some notation.
Let $A$ be either a bounded Lipschitz domain in $\re^n$ or
$A=\partial \Omega$ where $\Omega$ is a bounded domain of
$\re^{n+1}$ with Lipschitz boundary. Let $M \subset A$ be
an open set (relative to $A$) with Lipschitz boundary  (relative to $A$) $\Gamma =
\partial M$.
We define the following two sets:
\begin{equation}\label{Bfrac}
B_{\rm frac}=\begin{cases}A\times A & \mbox{if}\;\; 0<s\leq1/2\\
 M \times M & \mbox{if}\;\; 1/2<s<1,
\end{cases}
\end{equation}
and
\begin{equation}\label{Bweig}
B_{\rm weig}=\begin{cases}
 (A \setminus M) \times (A \setminus M)  &  \mbox{if}\;\; 0<s\leq 1/2\\
(A \setminus M) \times A   &  \mbox{if}\;\; 1/2<s<1.

\end{cases}
\end{equation}
For every $z$, we denote $d_M(z)$
the Euclidean distance (either in $\re^n$ or in $\re^{n+1}$) from the point $z$ to the set $M$. Set
$$a:=1-2s\in (-1,1).$$

In Theorem \ref{extension-s} we establish that, for any given function $w$ defined on all $\partial \Omega$,  there exists an extension $\widetilde w$ of $w$ to $\Omega$ whose $H^1(\Omega,{d_M^a})$-norm is controlled by a combination of a
$H^s$-norm and a $H^{1/2}(\cdot,{d_M^{a}})$-norm of its trace $w$. If $h$ is a weight (that is, a nonnegative function here), 
we indicate with
$H^s(\partial\Omega,h)$ the weighted Sobolev space of functions $w$
such that $$\int_{\partial\Omega} h(z) w(z)^2 d\mathcal H^n(z)  
+ \int_{\partial \Omega}\int_{\partial \Omega}h(z)\:\frac{|w(z)-w(\overline{z})|^2}{|z-\overline{z}|^{n+2s}} d\mathcal H^n(z)
d\mathcal H^n({\overline{z}})< +\infty.$$

Later we will apply our results in the case 
$$
\Omega=C_1,\ A=\partial C_1, \ M=B_1\times \{0\}, \ \text{ and } \ h=d_M^{1-2s}(x,\lambda)={\rm dist}^{1-2s} ((x,\lambda),M)=\lambda^{1-2s}.
$$

For a general domain the result is the following.

\begin{teo}\label{extension-s}
Let $\Omega$ be a bounded domain of $\re^{n+1}$ with Lipschitz
boundary $\partial \Omega$ and $M\subset \partial \Omega$ an open subset (relative to $\partial
\Omega$) with Lipschitz boundary (relative to $\partial
\Omega$). For $z\in \re^{n+1}$, let $d_M(z)$ denote
the Euclidean distance from the point $z$ to the set $M$. Let $s\in (0,1)$.

Then, for every $w\in C(\partial \Omega)$ there exists an extension $\widetilde{w}$ of $w$ defined in $\overline\Omega$ which
belongs to $C^1( \Omega)\cap C(\overline{\Omega})$ and such that

\begin{eqnarray}\label{ext-s}
&&\hspace{-3em}\int_{\Omega}d_M(z)^{1-2s}|\nabla
\widetilde{w}|^2dz \leq \nonumber \\
&&
\leq C||w||^2_{L^2(\partial \Omega)} +C
\int\int_{B_{\rm frac}}\frac{|w(z)-w(\overline{z})|^2}{|z-\overline{z}|^{n+2s}}d\mathcal H^n(z) d\mathcal H^n({\overline{z}})\\ 
&&\hspace{1em}
 + C\int\int_{B_{\rm weig}}
d_M(z)^{1-2s}\frac{|w(z)-w(\overline{z})|^2}{|z-\overline{z}|^{n+1}}d\mathcal H^n(z)d\mathcal H^n({\overline{z}})\nonumber,
\end{eqnarray} where
$B_{\rm frac}$ and $B_{\rm weig}$ are defined, respectively, in \eqref{Bfrac} and \eqref{Bweig} with $A=\partial \Omega$, and
$C$ denotes a positive constant depending only on $\Omega$, $M$, and $s$. 
\end{teo}

We have used the notations $B_{\rm frac}$ and $B_{\rm weig}$ to indicate, respectively, the set in which we compute the $H^s$-norm of $w$ and the set in which we compute the weighted $H^{1/2}(\cdot,{d_M^a})$-norm of $w$.
\begin{oss} \label{s-ext}
We denote by $\overline w$ the $s$-extension of $w$ in $\Omega$. Since $\overline w$ is the extension of $w$ in $\Omega$ which minimizes the quantity
$$\int_{\Omega}d_M(z)^{1-2s}|\nabla \widetilde w|^2dz,$$
then inequality \eqref{ext-s} holds, in particular, with $\widetilde w$ replaced by $\overline w$.
\end{oss}

In two articles \cite{N2, N1}, Nekvinda treated some extension and trace problems for functions belonging to 
fractional Sobolev spaces, but his results are not applicable in our situation.
In \cite{N1}, the author proved an extension theorem for functions belonging to $H^s(M)$, where $M$ is, 
as before, a subset of $\partial \Omega$. More precisely he proved that if $w\in H^s(M)$ then there 
exists an extension $\widetilde w$ of $w$ in $\Omega$ such that
$$\int_{\Omega}d_M (z)^{1-2s}|\nabla \widetilde w|^2dz \leq C||w||^2_{H^s(M)}.$$
Notice that this theorem gives an extension for a function defined only on $M$. There is no control
on the extension near $\partial\Omega\setminus M$ ---a control that we require. Instead,  
in \cite{N2}, he considered the case of a function $w$ defined on $\partial \Omega \setminus M$ and 
established that there exists an extension $\widetilde w$ of $w$ in $\Omega$ with $H^1(\cdot,{d_M^a})$-norm 
controlled by some weighted fractional norm of $w$ in $\partial \Omega \setminus M$.
In our situation, we need an extension result to all of $\Omega$ for functions $w$ defined on all of $\partial \Omega$.

We conclude giving a key result in the proof of Theorem
\ref{energy-est-s} on energy estimates for minimizers. It will give control on the $H^s$ double integrals above on
$A:=\partial\Omega$ knowing the following assumptions on the function $w$ defined on $A$.  

Let $A$, $M\subset A$, $\Gamma =\partial M$, $B_{\rm frac}$, and $B_{\rm weig}$ be as in \eqref{Bfrac} and \eqref{Bweig}. 
Let $D$ denote all tangential derivatives to $A$ and $d_{\Gamma}(z)$ denote the Euclidean distance from the point
$z$ to the
set $\Gamma$ (either in $\re^n$ or in $\re^{n+1}$). 
Note that here we deal with the distance $d_\Gamma=d_{\partial M}$ to $\Gamma=\partial M$, in contrast with the distance $d_M$ to $M$ appearing  in the weighted energies \eqref{ext-s} above and \eqref{s} below.

In what follows we will assume that, for some constant $c_s$, $w$ satisfies these conditions:
\begin{itemize}
\item for  $s\in (0,1/2]$,
\begin{equation}\label{grad-1A}
|Dw(z)|\leq \begin{cases}
\displaystyle\frac{c_s}{\varepsilon}\left(\frac{d_\Gamma(z)}{\varepsilon}\right)^{2s-1} & \mbox{if} \; z\in A\;\;\mbox{and}\;\;d_\Gamma(z)\leq\varepsilon\\
\displaystyle\frac{c_s}{d_\Gamma(z)}& \mbox{if} \; z\in A\;\;\mbox{and}\;\;d_\Gamma(z)>\varepsilon;\end{cases}\end{equation}
\item for  $s\in (1/2,1)$,
\begin{equation}\label{grad-2M}
|Dw(z)|\leq \begin{cases}
\displaystyle\frac{c_s}{\varepsilon} & \mbox{if} \; z\in A\;\;\mbox{and}\;\;d_\Gamma(z)\leq\varepsilon\\
\displaystyle\frac{c_s}{d_\Gamma(z)}& \mbox{if} \; z\in A\;\;\mbox{and}\;\;d_\Gamma(z)>\varepsilon.\end{cases}\end{equation}
\end{itemize}

Later we will use this result with 
$$
A=\partial C_1,\ M=B_1\times \{\lambda=0\}, \ \textrm{ and } \Gamma=\partial B_1 \times
\{\lambda=0\}.
$$

In more general geometries the result is the following. 

\begin{teo}\label{key-s}
Let $A$ be either a bounded Lipschitz domain in $\re^n$ or
$A=\partial \Omega$ where $\Omega$ is a bounded domain of
$\re^{n+1}$ with Lipschitz boundary. Let $M \subset A$ be
an open set (relative to $A$) with Lipschitz boundary (relative to $A$) $\Gamma =
\partial M$.  Let $\varepsilon\in (0,1/2)$ and $s \in (0,1)$.

Suppose that, for some constant $c_s$, $w:A \rightarrow \re$ is a Lipschitz function such
that
\begin{equation}\label{bound_w-s}|w(z)|\leq c_s\end{equation} 
and that $w$ satisfies \eqref{grad-1A} and \eqref{grad-2M} for almost every $z\in A$.

Then,
\begin{equation}\label{s}
\begin{split}
&
\Psi_s(w):=||w||^2_{L^2(A)}+\int\int_{B_{\rm frac}}\frac{|w(z)-w(\overline
z)|^2}{|z-\overline z|^{n+2s}}d\mathcal H^n(z) d\mathcal H^n({\overline z}) \\ 
&\hspace{.1cm}+
\int\int_{B_{\rm weig}}d_M(z)^{1-2s}
\frac{|w(z)-w(\overline z)|^2}{|z-\overline z|^{n+1}}d\mathcal H^n(z)
d\mathcal H^n({\overline z})
  \leq C\int_{\varepsilon}^1 \rho^{-2s}d\rho,
\end{split}
\end{equation}
where the sets $B_{\rm frac}$ and $B_{\rm weig}$ are defined in \eqref{Bfrac} and \eqref{Bweig}, and 
$C$ denotes a positive constant depending on
$A$, $M$, $n$, $s$, and $c_s$.

As a consequence, for some constant $C$ depending on the same quantities as before, we have
\begin{eqnarray*}
\Psi_s(w)
\leq  \begin{cases} C & \mbox{if}\;\;0<s<1/2, \\
C |\log \varepsilon| & \mbox{if}\;\;s=1/2, \\
                           \displaystyle C \varepsilon^{1-2s} & \mbox{if}\;\;1/2<s<1.
                            \end{cases}
\end{eqnarray*}
\end{teo}

In our case ($A=\partial C_1$,
$M=B_1\times \{\lambda=0\}$, $\Gamma=\partial B_1 \times
\{\lambda=0\}$), the constant $C$ in
\eqref{s} only depends on $n$, $s$, and $c_s$.

\begin{oss}\label{oss-grad} 
In the proof of Theorem \ref{energy-est-s} the following gradient estimates for every bounded solution $v$ of 
problem \eqref{eq2-s} will be of utmost importance. Let $f\in C^{1,\gamma}$ for some $\gamma>\max\{0,1-2s\}$.
Then, every bounded solution $v$ of \eqref{eq2-s} satisfies, for some constant $c_s$,
\begin{equation}\label{grad-s}
\begin{cases}
\displaystyle |\nabla_x v(x,\lambda)|\leq c_s &\mbox{for
every}\;\;x \in \re^{n}\;\;\mbox{and}\;\;\lambda\geq 0;\vspace{0.6em}\\
\displaystyle |\nabla v(x,\lambda)|\leq c_s/\lambda &\mbox{for
every}\;\;x \in \re^{n}\;\;\mbox{and}\;\;\lambda> 0;\vspace{0.6em}\\
\displaystyle |\lambda^{1-2s}\partial_{\lambda}v|\leq c_s & \mbox{for
every}\;\;x\in\re^n\;\;\mbox{and}\;\;\lambda>0.\end{cases}\end{equation}
For the bound $|\nabla_x v(x,0)|\leq c_s$ for $x\in \re^n$, see Silvestre \cite{Sil}, Lemmas 2.8 and 2.9 in  \cite{Sil}.
In this inequality the constant $c_s$ is uniformly bounded for $s$ away from zero, but not as $s\rightarrow 0$. 
For this reason the constant $C$ in our energy estimate \eqref{energy-s} is not uniform for $s$ close to zero.
Using the maximum principle we can extend the bound $|\nabla_x v(x,0)|\leq c_s$ to every $\lambda >0$ and deduce 
$|\nabla_x v(x,\lambda)|\leq c_s$ for every $x\in \re^n$ and $\lambda \geq 0$; see Proposition 4.6 of \cite{C-Si1}.

The bound $|\nabla v(x,\lambda)|\leq c_s /\lambda$ for every $x\in \re^n$ and $\lambda>0$ follows, after rescaling, 
by interior elliptic estimates, since equation \eqref{eq2-s} is uniformly elliptic for $\lambda>0$; see Proposition 4.6 of \cite{C-Si1}. In this bound the constant $c_s$ is also bounded as $s\uparrow 1$.

Finally, the last bound $|\lambda^{1-2s}\partial_\lambda v(x,\lambda)|\leq c_s$ for every $x\in \re^n$ and $\lambda \geq 0$ 
is established using that the function $\widetilde v=\lambda^{1-2s}\partial_{\lambda}v$ satisfies the dual problem 
(with Dirichlet boundary condition)
$$\begin{cases}
   \mbox{div}(\lambda^{2s-1}\nabla \widetilde v)=0 &\mbox{in}\:\:\re^{n+1}_+\\
\displaystyle\widetilde v=-\frac{f(u)}{d_s} & \mbox{on}\:\:\partial\re^{n+1}_+;
  \end{cases}
$$
see also Proposition 4.6 of \cite{C-Si1}. In this last gradient estimate, the constant $c_s$ is uniformly bounded for $s$ away from $1$ but not as $s\rightarrow 1$ (since $d_s/(1-s)\rightarrow 1$ as $s\rightarrow 1$). For this reason, the constant $C$ in our energy estimate \eqref{energy-s} blows up for $s$ close to $1$.
\end{oss}

The paper is organized as follows:
\begin{itemize}
\item In section 2 we prove the energy estimate for layer
solutions of bistable type equations, using a simple argument
introduced by Ambrosio and the first author \cite{AC}.
\item In section 3 we establish a monotonicity formula for the energy functional associated to problem \eqref{eq2-s}.
\item In section 4 we give the proof of the extension Theorem \ref{extension-s}
 and of the key Theorem \ref{key-s}.
\item In section 5 we prove Theorem \ref{energy-est-s}.
\item In section 6 we establish energy estimates for bounded monotone solutions
in $\re^3$ (Theorem \ref{energy-dim3-s}).
\item In section 7 we prove the 1-D symmetry result, that is
Theorem \ref{degiorgi-s}.
\end{itemize}

\section{Energy estimate for monotone solutions of bistable equations}

In this section we consider potentials $G(u)=\int_u^1f$ satisfying
the following hypothesis: \begin{equation}\label{h2-s}G\geq 0=G(\pm
1) \quad \mbox{in}\; \re \quad \mbox{and}\;\; G>0 \quad \mbox{in}\;
(-1,1).\end{equation} An example is
$G(u)=\frac{1}{4}(1-u^2)^2$. In this case the nonlinearity is given by
$f(u)=u-u^3$. In \cite{C-Si1, C-Si2} it is proved the existence of a 1-D layer solution for every nonlinearity satisfying \eqref{h2-s}.

In the sequel we consider the energy
$$\mathcal{E}_{s,C_R}(v)=d_s\int_{C_R}\frac{1}{2}\lambda^{1-2s}|\nabla v|^2dx d\lambda
+ \int_{B_R}G(v(x,0))dx.$$

In the following theorem we establish energy estimates for monotone solutions of \eqref{eq1-s} such that $\lim_{x_n\rightarrow
+\infty}u(x',x_n)=1$ for all $x'\in \re^{n-1}$, assuming that the potential $G$ satisfies \eqref{h2-s}. We point out that here the constant appearing in our energy estimate is uniformly bounded as $s\uparrow 1$ (but not as $s\downarrow 0$). Recall that we have defined the cylinder $C_R=B_R\times (0,R)$, where $B_R$ is the ball centered at $0$ and of radius $R$ in $\re^n$.

\begin{teo}\label{energy-layer-s}
Let $f$ be a $C^{1,\gamma}$ function, with $\gamma >\max\{0,1-2s\}$. Suppose that $G(u)=\int_u^1f$
satisfies \eqref{h2-s}. Let $u$ be a solution of problem
\eqref{eq1-s} in $\re^n$, with $|u|<1$, and let $v$ be the $s$-extension
of $u$ in $\re^{n+1}_+$.

Given any $s_0 \in (0, 1/2)$, assume that $s_0<s<1$.

Suppose that
\begin{equation}\label{hp-monot-s}u_{x_n}>0 \;\;\mbox{in}\;\; \re^n
\end{equation} and
\begin{equation}\label{limit-s}
 \lim_{x_n\rightarrow
+\infty}u(x',x_n)=1\;\;\mbox{for all}\:\;x'\in
\re^{n-1}.\end{equation}

Then, $v$ satisfies the energy estimate \eqref{energy-s} for every $R>2$, where the constant $C$ depends only on $n$, $s_0$, and $||f||_{C^{1,\gamma}([-1,1])}$. 
In particular, for some constant $C$ depending on these three quantities, we have
\begin{equation}
\mathcal E_{s,C_R}(v) \leq \frac{ C}{1-2s} R^{3-2s} \quad
\mbox{if}\;\;s_0<s<1/2\end{equation}
\begin{equation}
\mathcal E_{s,C_R}(v) \leq CR^{2}\log R \quad
\mbox{if}\;\;s=1/2\end{equation}
\begin{equation}
\mathcal E_{s,C_R}(v) \leq \frac{C}{2s-1} R^{2}\quad
\mbox{if}\;\;1/2<s<1.
\end{equation}

\end{teo}

\begin{proof}
As in \cite{C-Cinti1}, the proof uses an argument
found by Ambrosio and the first author \cite{AC} to prove an energy estimate for
layer solutions of the analogue problem $-\Delta u=f(u)$. This method is
based on sliding the function $v$ in the direction $x_n$.
Consider the function
$$v^t(x,\lambda):=v(x',x_n+t,\lambda)$$ defined for $(x,\lambda)=(x',x_n,\lambda)\in
\re^{n+1}_+$ and $t\in \re$. For each $t$ we have
\begin{equation}
\begin{cases}
\text{div}(\lambda^{1-2s}\nabla v^t)=0& \text{in}\; \re_{+}^{n+1},\\
-d_s\lambda^{1-2s}\partial_{\lambda}v^t=f(v^t)& \text{on}\;
\re^{n}=\partial \re_{+}^{n+1}.
\end{cases}
\end{equation}
Now we
use the first two gradient estimates in \eqref{grad-s} for the solution $v$ of problem \eqref{eq2-s}. We have that for every $t$, $|v^t|<1$ and
\begin{equation}\label{grad-x}
|\nabla_x v^t(x,\lambda)|\leq c_s\:\:\mbox{for
every}\;\;x\in\re^n\;\;\mbox{and}\;\;\lambda\geq 0,\end{equation}
\begin{equation}\label{grad-lambda>1}
|\nabla v^t(x,\lambda)|\leq \frac{c_s}{\lambda}\:\:
\mbox{for every}\;\;x\in\re^n\;\;\mbox{and}\;\;\lambda>0, \end{equation}
 where the constant $c_s$ is uniformly bounded for $s$ away from $0$ (see Remark \ref{oss-grad}).
 In fact, $c_s$ depends only on $n$, $s_0$, and $||f||_{C^{1,\gamma}([-1,1])}$ (see Proposition 4.6 of \cite{C-Si1}).
In addition (see Lemma 4.8 of \cite{C-Si1}) 
\begin{equation}\label{limv-s}\lim_{t\rightarrow +\infty} \left\{|v^t(x,\lambda)-1|+|\nabla v^t(x,\lambda)|\right\}=0\end{equation} for all $x\in
\re^n$ and all $\lambda\geq 0$.

Note that hypothesis (\ref{hp-monot-s}) and the maximum principle lead to $v_{x_n}>0$ in $\re^{n+1}_+$. Thus,
denoting the derivative of $v^t(x,\lambda)$ with respect to $t$ by
$\partial_t v^t(x,\lambda)$, we have 
$$\partial_t v^t(x,\lambda)=v_{x_n}(x',x_n+t,\lambda)>0\quad
\mbox{for all}\;\;x \in \re^n,\;\lambda\geq 0.$$
By \eqref{limv-s}, we have that
\begin{equation}
\label{limener}
\lim_{t\rightarrow
+\infty}\mathcal{E}_{s,C_R}(v^t)=0.
\end{equation}

Next, we bound the
derivative of $\mathcal{E}_{s,C_R}(v^t)$ with respect to $t$. Recall that we have
set $a=1-2s$. We
use that $v^t$ is a bounded solution of problem (\ref{eq2-s}), the bounds \eqref{grad-x}, \eqref{grad-lambda>1} for the derivatives of $v^t$, and the crucial fact
that $\partial_t v^t>0$. Let $\nu$ denote the exterior normal to
the lateral boundary $\partial B_R\times (0,R)$ of the cylinder
$C_R$.

We have
\begin{equation*}
\begin{split}
\partial_t\mathcal{E}_{s,C_R}(v^t)&= d_s\int_0^R d\lambda
\int_{B_R}dx\lambda^a\nabla v^t\cdot \nabla (\partial_t v^t)
+\int_{B_R}dxG'(v^t (x,0))\partial_t v^t(x,0)\\
&= d_s\int_0^Rd\lambda \int_{\partial B_R}d\mathcal H^{n-1}\lambda^a\frac{\partial
v^t}{\partial \nu}
\partial_t v^t + d_s\int_{B_R\times \{\lambda=R\}}dx\lambda^a\frac{\partial
v^t}{\partial\lambda}
\partial_t v^t(x,R) \\
&\geq -Cd_s\int_0^R d\lambda \frac{\lambda^a}{1+\lambda}\int_{\partial B_R}d\mathcal H^{n-1} \partial_t v^t- Cd_s{R^{-2s}}\int_{B_R\times \{\lambda=R\}}dx \partial_t
v^t(x,R),\end{split}\end{equation*} 
where in the last inequality we have used both gradient bounds \eqref{grad-x} and \eqref{grad-lambda>1} for the first term (since $\frac{\partial v^t}{\partial \nu}$ is a horizontal derivative) and the bound \eqref{grad-lambda>1} for the second term. We recall that here $C$ is a constant uniformly bounded for $s$ away from $0$.
Hence, for every $T>0$, we have
\begin{eqnarray*}
\mathcal{E}_{s,C_R}(v)&=&\mathcal{E}_{s,C_R}(v^T)-\int_0^T\partial_t
\mathcal{E}_{s,C_R}(v^t)dt \\
&\leq& \mathcal{E}_{s,C_R}(v^T)+ Cd_s\int_0^Tdt \int_0^R d\lambda \frac{\lambda^a}{1+\lambda}\int_{\partial B_R}d\mathcal H^{n-1} \partial_t v^t \\
&&\hspace {0.7em} + Cd_sR^{-2s}\int_0^T dt\int_{B_R\times
\{\lambda=R\}}dx \partial_t v^t(x,R)  \\
&=& \mathcal{E}_{s,C_R}(v^T)+ Cd_s\int_{\partial B_R}d\mathcal H^{n-1} \int_0^R d\lambda\frac{\lambda^a}{1+\lambda}  \int_0^T dt \partial_t v^t(x,\lambda)\\
&&\hspace {0.7em} + Cd_sR^{-2s}\int_{B_R\times
\{\lambda=R\}}dx \int_0^T dt\partial_t v^t(x,R)  \\
&=& \mathcal{E}_{s,C_R}(v^T)+Cd_s\int_{\partial B_R}d\mathcal H^{n-1} \int_0^R d\lambda\frac{\lambda^a}{1+\lambda} (v^T-v^0)(x,\lambda)\\
&&\hspace {0.7em} + Cd_sR^{-2s}\int_{B_R\times
\{\lambda=R\}}dx(v^T-v^0)(x,R) \\
&\leq&\mathcal{E}_{s,C_R}(v^T) + Cd_sR^{n-1}\int_0^R \frac{\lambda^a}{1+\lambda} d\lambda +
 Cd_sR^{n-2s}\\
&\leq&  \mathcal{E}_{s,C_R}(v^T) + Cd_sR^{n-1}\left(\int_0^1\lambda^{1-2s}d\lambda+
\int_1^R \lambda^{-2s}d\lambda \right)+Cd_sR^{n-2s}.
\end{eqnarray*}
Using the change of variables $\lambda=\rho R$, we have
$ \int_0^1\lambda^{1-2s}d\lambda +
\int_1^R \lambda^{-2s}d\lambda
= \frac{1}{2(1-s)}+R^{1-2s}\int_{1/R}^1
\rho^{-2s}d \rho$.
Thus, using that $d_s\approx 2(1-s)$ as $s\uparrow 1$ and distinguishing the three cases $s_0<s<1/2$, $s=1/2$, and $1/2<s<1$, we conclude
\begin{equation*}
\begin{split}
\mathcal E_{s,C_R}(v) & \leq \mathcal{E}_{s,C_R}(v^T) + CR^{n-1}+C(1-s)R^{n-2s}\int_{1/R}^1\rho^{-2s}d\rho+C(1-s)R^{n-2s}\\
&\leq \mathcal E_{s,C_R}(v^T) + CR^{n-2s}\int_{1/R}^1\rho^{-2s}d \rho.\end{split}\end{equation*}

Letting $T \rightarrow + \infty$ and using \eqref{limener}, we obtain the desired estimate.
\end{proof}

\section{Monotonicity formula}
In this section we establish a monotonicity formula for the energy
functional associated to problem \eqref{eq2-s}. More precisely, we
prove that for every solution $v$ of problem \eqref{eq2-s} the quantity $\mathcal{E}_{\widetilde
B_{R}^+}(v)/R^{n-2s}$ is nondecreasing in $R$, where $\widetilde
B_{R}^+$ is the positive half ball in $\re^{n+1}_+$ centered at
$0$ and of radius $R$. 
From this result we deduce that our
energy estimate \eqref{energy-s1} is sharp when $0<s<1/2$ (see Remark \ref{R-mon})

In the following lemma we prove a Pohozaev identity for solutions
of problem \eqref{eq2-s} which will be important in the proof of
our monotonicity formula. We use the following notation:
$$\widetilde
B_{R}^+=\{(x,\lambda)\in \re^{n+1}_+:|(x,\lambda)|<R\}\:\:\mbox{and}\:\:\partial^+\widetilde B_R^+=\partial\widetilde B_R^+\cap\{\lambda>0\}.$$

 \begin{lem}\label{pohozaev}
Let $s\in (0,1)$, $f$ be any $C^{1,\gamma}$ nonlinearity with
$\gamma>\max\{0,1-2s\}$, and
suppose that $v$ is a bounded solution of problem \eqref{eq2-s}.

Then, for every $R>0$
\begin{eqnarray*}
&&\frac{n-2s}{2}\int_{\widetilde B_{R}^+}\lambda^{1-2s}|\nabla v|^2 dx d\lambda +
n\int_{B_R\times\{0\}}d_s^{-1} G(v) dx =\\
&&\hspace{1em}=\frac{R}{2}\int_{\partial^+\widetilde
B_{R}^+}\lambda^{1-2s}|\nabla v|^2
d\mathcal{H}^n-R\int_{\partial^+\widetilde
B_{R}^+}\lambda^{1-2s}\left(\frac{\partial v}{\partial
\nu}\right)^2 d\mathcal{H}^n + \\
&&\hspace{2em} +R\int_{\partial B_R\times\{0\}} d_s^{-1} G(v)
d\mathcal{H}^{n-1},\end{eqnarray*} where $d\mathcal{H}^{n-1}$ and $d\mathcal{H}^n$ denote respectively
the $(n-1)$-dimensional and $n$-dimensional Hausdorff measures and $\frac{\partial
v}{\partial \nu}$ denotes the outer normal derivative of
$v$ on $\partial^+\widetilde B_{R}^+$.
\end{lem}

\begin{proof} Set $z=(x,\lambda)$. Multiplying the equation $\mbox{div}(\lambda^{1-2s}\nabla v)=0$ by
$\langle z,\nabla v\rangle= \langle (x,\lambda),\nabla v\rangle$,
we have
\begin{eqnarray*}
0&=&\mbox{div}(\lambda^{1-2s}\nabla v)\langle z,\nabla v\rangle\\
&=&\mbox{div}(\lambda^{1-2s}\nabla v\langle z,\nabla
v\rangle)-\lambda^{1-2s}\left\{|\nabla v|^2+\langle z, \nabla
\left(\frac{|\nabla v|^2}{2}\right)\rangle \right\}.\end{eqnarray*}
Now, observe that
\begin{eqnarray*}
\lambda^{1-2s}\langle z, \nabla \left(\frac{|\nabla
v|^2}{2}\right)\rangle&=&\mbox{div}\left\{\lambda^{1-2s}z\left(\frac{|\nabla
v|^2}{2}\right)\right\} \\
&&\hspace{1em}-\lambda^{1-2s}(n+1)\frac{|\nabla
v|^2}{2}-(1-2s)\lambda^{1-2s}\frac{|\nabla v|^2}{2}\\
&=& \mbox{div}\left\{\lambda^{1-2s} z
\left(\frac{|\nabla
v|^2}{2}\right)\right\}-\frac{n+2-2s}{2}\lambda^{1-2s}|\nabla
v|^2.\end{eqnarray*} Thus, we obtain
\begin{eqnarray*}
\mbox{div}\left\{\lambda^{1-2s}\left(\nabla v\langle z,\nabla
v\rangle-z\frac{|\nabla v|^2}{2}\right)\right\}
+ \frac{n-2s}{2}\lambda^{1-2s}|\nabla v|^2=0.\end{eqnarray*} Next,
we integrate by parts on $\widetilde B_R^+$:
\begin{eqnarray*}
&&\int_{\partial^+\widetilde B_R^+}\lambda^{1-2s}\langle
\nu,\nabla v \rangle \langle z,\nabla v \rangle d\mathcal H^n +
\int_{B_R\times\{0\}}-\lambda^{1-2s}\partial_{\lambda}v\langle x,\nabla_x v
\rangle dx \\
&&\hspace{2em}-\frac{1}{2}\int_{\partial^+\widetilde
B_R^+}\lambda^{1-2s}|\nabla v|^2\langle z,\nu\rangle d\mathcal
H^n+\frac{n-2s}{2}\int_{\widetilde B_R^+}\lambda^{1-2s}|\nabla
v|^2dx d\lambda=0,\end{eqnarray*}
where $\nu$ denotes the outer unit normal vector to $\partial^+\widetilde B_R^+$.

Now, we use that $z=R\nu$ on $\partial^+\widetilde B_R^+$ and
$-d_s\lambda^{1-2s}\partial_{\lambda}v=f(v)$ on $B_R$, to get
\begin{eqnarray*}
&&R\int_{\partial^+\widetilde {B}_R^+}
\lambda^{1-2s}\left(\frac{\partial v}{\partial \nu}\right)^2
d\mathcal H^n+\int_{B_R\times\{0\}}d_s^{-1}f(v)\langle x, \nabla_x v\rangle dx \\
&&\hspace{2em}-\frac{R}{2}\int_{\partial^+\widetilde {B}_R^+}
\lambda^{1-2s}|\nabla v|^2 d\mathcal H^n +
\frac{n-2s}{2}\int_{\widetilde {B}_R^+} \lambda^{1-2s} |\nabla
v|^2 dz=0.
\end{eqnarray*}
We conclude the proof, observing that, on $\{\lambda=0\}$,
\begin{eqnarray*}
&&\int_{B_R}f(v)\langle x, \nabla_x v\rangle dx=-\int_{B_R}
\langle x, \nabla_x G(v) \rangle dx
=\int_{B_R}\left(-\mbox{div}(xG(v))+nG(v)\right)dx
\\&&\hspace{2em} =n\int_{B_R}G(v)dx -R\int_{\partial B_R}G(v)
d\mathcal H^{n-1}.
\end{eqnarray*}
\end{proof}

We can now prove the monotonicity formula.
\begin{prop}\label{mon-formula}
Let $s\in (0,1)$, $f$ be any $C^{1,\gamma}$ nonlinearity with
$\gamma>\max\{0,1-2s\}$, and
suppose that $v$ is a bounded solution of problem \eqref{eq2-s}, and that
$G(t)\geq 0$ for every $t\in \re$.

Then, the function
$$\phi(R)=\frac{1}{R^{n-2s}}\left\{\frac{d_s}{2}\int_{\widetilde
B_R^+}\lambda^{1-2s}|\nabla v|^2 dx d\lambda +\int_{B_R\times\{0\}}G(v)dx\right\}$$ is
a nondecreasing function of $R>0$.
\end{prop}
\begin{proof} We have that
\begin{eqnarray*}
&&\phi'(R)=-\frac{(n-2s)d_s}{2R^{n-2s+1}}\int_{\widetilde
B_R^+}\lambda^{1-2s}|\nabla v|^2 dx d\lambda+ \frac{d_s}{2R^{n-2s}}
\int_{\partial^+\widetilde B_R^+}\lambda^{1-2s}|\nabla v|^2 d\mathcal H^n\\
&&\hspace{2em}-\frac{n-2s}{R^{n-2s+1}}\int_{B_R\times \{0\}}G(v)dx+\frac{1}{R^{n-2s}}\int_{\partial
B_R\times \{0\}}G(v)d\mathcal H^{n-1}.\end{eqnarray*} Using the Pohozaev identity of Lemma \ref{pohozaev}, we get
$$\phi'(R)=\frac{d_s}{R^{n-2s}}\int_{\partial^+\widetilde
B_R^+}\lambda^{1-2s}\left(\frac{\partial v}{\partial \nu}\right)^2
d\mathcal H^n + \frac{2s}{R^{n-2s+1}}\int_{B_R\times \{0\}}G(v)dx \geq 0.$$
\end{proof}
\begin{oss} \label{R-mon}
 Since
$$\mathcal E_{\widetilde B^+_{R}}(v)\leq  \mathcal E_{C_{R}}(v),$$ 
Proposition \ref{mon-formula} gives that, for every bounded solution $v$ of problem \eqref{eq2-s} which is not identically zero, the following lower bound holds:
\begin{equation}\label{lower}d_s\int_{C_R}\frac{1}{2}\lambda^{1-2s}|\nabla v|^2dx
d\lambda+\int_{B_R}\{G(u)-c_u\}dx \geq c_1R^{n-2s},\end{equation} for some constant $c_1>0$ depending on $v$. 
Note that Theorem \ref{energy-est-s} establishes $\mathcal E_{C_R}(v)\leq CR^{n-2s}$ for every $s \in (0,1/2)$. Thus, this bound is sharp as a consequence of \eqref{lower}.

\end{oss}

\section{$H^{s}$ estimate}
In this section we recall some definitions and properties of
the spaces $H^{s}(\re^n)$ and $H^{s}(\partial \Omega)$, where
$\Omega$ is a bounded subset of $\re^{n+1}$ with Lipschitz
boundary $\partial \Omega$ (see
\cite{LM}).

$H^{s}(\re^n)$ is the space of functions $u\in L^2(\re^n)$ such
that
$$\int_{\re^n}\int_{\re^n}\frac{|u(x)-u(\overline{x})|^2}{|x-\overline{x}|^{n+2s}}dx
d\overline{x}< +\infty,$$ equipped with the norm

$$||u||_{H^{s}(\re^n)}=\left(||u||^2_{L^2(\re^n)}+ \int_{\re^n}\int_{\re^n}\frac{|u(x)-u(\overline{x})|^2}{|x-\overline{x}|^{n+2s}}dx
d\overline{x}\right)^{\frac{1}{2}}.$$

As in section 3 of \cite{C-Cinti1}, using a family of charts and a
partition of unity, we can define the space $H^s(\partial
\Omega)$, where $\Omega$ is a bounded subset of $\re^{n+1}$ with
Lipschitz boundary.
We use the same notations of \cite{C-Cinti1}.

Consider an atlas
$\{(O_j,\varphi_j), j=1,...,m\}$ where $\{O_j\}$ is a family of
open bounded sets in $\re^{n+1}$ such that  $\{O_j\cap\partial
\Omega;j=1,...,m\}$ cover $\partial \Omega$. The functions $\varphi_j$ are  corresponding bilipschitz
diffeomorphisms such that
\begin{itemize}
\item $\varphi_j:O_j\rightarrow U:=\{(y,\mu)\in
\re^{n+1}:|y|<1,\:-1<\mu<1\}$, \item

$\varphi_j:O_j\cap \Omega \rightarrow U^+:=\{(y,\mu)\in
\re^{n+1}:|y|<1,\:0<\mu<1\}$, and

\item $\varphi_j:O_j\cap\partial \Omega \rightarrow
\{(y,\mu)\in \re^{n+1}:|y|<1,\mu=0\}$.

\end{itemize}

Let $\{\alpha_j\}$ be a partition of unity on $\partial \Omega$
such that $0\leq \alpha_j \in C^{\infty}_c(O_j),
\:\sum_{j=1}^m \alpha_j=1$ on $\partial \Omega$. If $u$ is a function on $\partial
\Omega$, decompose $u=\sum_{j=1}^m u\alpha_j $ and define the function
$$(u\alpha_j)\circ\varphi_j^{-1}(y,0):=(u\alpha_j
)(\varphi_j^{-1}(y,0)),\quad \mbox{for every}\;(y,0)\in U\cap
\{\mu=0\}.$$

Since $\alpha_j$ has compact support in $O_j$,
the function $(u\alpha_j)\circ\varphi_j^{-1}(\cdot,0)$ has compact support in
$U\cap \{\mu=0\}$ and therefore we may consider
$((u\alpha_j)\circ \varphi_j^{-1})(\cdot,0)$ to be defined in $\re^n$ extending it
by zero out of $U\cap \{\mu=0\}$.  Now we define
$$H^{s}(\partial \Omega):=\{u\,|\,(u\alpha_j)\circ\varphi_j^{-1}(\cdot,0)\in
H^{s}(\re^n),\:j=1,...,m\}$$ equipped with the norm

$$\left(\sum_{j=1}^m ||(u\alpha_j)\circ\varphi_j^{-1}(\cdot,0)||_{H^{s}(\re^n)}^2\right)^{\frac{1}{2}}.$$ Independently of the choice of the
system of local maps $\{O_j,\varphi_j\}$ and of the partition of
unity $\{\alpha_j\}$, these norms are all equivalent to
$$||u||_{H^{s}(\partial \Omega)}:=\left(||u||^2_{L^2(\partial
\Omega)}+\int_{\partial \Omega}\int_{\partial
\Omega}\frac{|u(z)-u(\overline{z})|^2}{|z-\overline{z}|^{n+2s}}d\mathcal H^n({z})
d\mathcal H^n({\overline{z}})\right)^{\frac{1}{2}}.$$

We can give now the proof of Theorem \ref{extension-s}.

\begin{proof}[Proof of Theorem \ref{extension-s}] Let $s\in (0,1)$.\\
\textbf{Case 1}: $\Omega=\re^{n+1}_+$.
We first consider the case of a half space
$\Omega=\re^{n+1}_+$ and $M=\{(x',x_n)\in \re^n:x_n<0\}$. Let
$\zeta$ be a bounded function belonging to $C(\re^n)$. Following the first
part of the proof of Proposition 3.1 in \cite{C-Cinti1}, we
consider a $C^{\infty}$ function  $K(x)$, defined on $\re^n$ with
compact support in $B_1$ and such that $\int_{\re^n} K(x)dx=1$.
Define $\widetilde{K}(x,\lambda)$ on $\re^{n+1}_+$ in the
following way:
$$\widetilde{K}(x,\lambda):=\frac{1}{\lambda^{n}}K\left(\frac{x}{\lambda}\right)$$
and finally define the extension $\widetilde{\zeta}$ as
\begin{equation}\label{zeta-s}\widetilde{\zeta}(x,\lambda)=\int_{\re^n}\widetilde{K}(x-\overline{x},\lambda)\zeta(\overline{x})d\overline{x}.\end{equation}
Note that, since $\int_{\re^n} \widetilde K(x,\lambda) dx=1$, we have
\begin{equation}\label{L2-s}
||\widetilde \zeta(\cdot,\lambda)||_{L^2(\re^n)}\leq ||\zeta||_{L^2(\re^n)}\quad \mbox{for every} \:\:\lambda \geq 0,\end{equation} and thus
\begin{equation}\label{L2w-s}
 \int_0^1 d\lambda \:\lambda^{1-2s}\int_{\re^n} dx |\widetilde \zeta(x,\lambda)|^2 \leq \frac{1}{2(1-s)} ||\zeta||^2_{L^2(\re^n)}.
\end{equation}

In \cite{C-Cinti1} (see the proof of Proposition 3.1), a simple calculation led to the following
estimate for the gradient of $\widetilde{\zeta}$:
\begin{equation}\label{1-s}
|\nabla \widetilde{\zeta}(x,\lambda)|^2\leq C
\int_{\{|x-\overline{x}|<\lambda\}}\frac{|\zeta(x)-\zeta(\overline{x})|^2}{\lambda^{n+2}}d\overline{x},\end{equation}
where $\nabla$ denotes the gradient with respect to $x$ and to $\lambda$, and $C$ depends only on $n$.

Since $M=\{(x',x_n)\in \re^n:x_n<0\}$, we have
$d_M(x,\lambda)=((x_n)_+^2+\lambda^2)^{1/2}$, where as usually
$(x_n)_+=\max \{x_n,0\}$.

Consider now,
separately, the two cases $0<s\leq 1/2$ and $1/2<s<1$.

If $0<s \leq 1/2$ then $a=1-2s \in [0,1)$ and we have that $d^a_M(x,\lambda)\leq (x_n)_+^a +\lambda^a$.
In the following computations $C$ will denote different positive constants which may depend on $s$.
Using \eqref{1-s}, we have
\begin{eqnarray*}
&&\hspace{-1em}\int_{\re^{n+1}_+}d_M^{a}(x,\lambda)|\nabla
\widetilde{\zeta}(x,\lambda)|^2dxd\lambda\leq 
\int_{\re^{n+1}_+}((x_n)_+^{a}+\lambda^{a})|\nabla
\widetilde{\zeta}(x,\lambda)|^2dxd\lambda\\
&& \hspace{4em}\leq
C\int_0^{+\infty}d\lambda
\int\int_{\{|x-\overline{x}|<\lambda\}}dx
d\overline{x}\:\frac{(x_n)_+^{a}+\lambda^{a}}{\lambda^{n+2}}|\zeta(x)-\zeta(\overline{x})|^2 \nonumber \\
&& \hspace{4em}\leq
C\int_0^{+\infty}d\lambda\int\int_{\{|x-\overline{x}|<\lambda\}}dx
d\overline{x}\:
\frac{1}{\lambda^{n+2-a}}|\zeta(x)-\zeta(\overline{x})|^2 \nonumber \\
&&\hspace{5em}+
C\int_0^{+\infty}d\lambda\int\int_{\{|x-\overline{x}|<\lambda\}}dx
d\overline{x}\:
\frac{(x_n)_+^a}{\lambda^{n+2}}|\zeta(x)-\zeta(\overline{x})|^2\nonumber \\
&&\hspace{4em}\leq
C\int_{\re^n}\int_{\re^n} dx
d\overline{x}|\zeta(x)-\zeta(\overline{x})|^2 [\lambda^{-n-1+a}]_{+\infty}^{|x-\overline x|} \\
&&\hspace{5em}+ C\int_{\re^n}\int_{\re^n} dx
d\overline{x}\:(x_n)_+^a|\zeta(x)-\zeta(\overline{x})|^2 [\lambda^{-n-1}]_{+\infty}^{|x-\overline x|}\\
&&\hspace{4em}\leq
C\int_{\re^n}\int_{\re^n}\frac{|\zeta(x)-\zeta(\overline{x})|^2}{|x-\overline{x}|^{n+2s}}dx
d\overline{x}\\
&&\hspace{5em}+C\int_{\re^n}\int_{\re^n}
(x_n)_+^{1-2s}\frac{|\zeta(x)-\zeta(\overline{x})|^2}{|x-\overline{x}|^{n+1}}dx
d\overline{x}.\end{eqnarray*}
Next, we observe that, in this bound, the last integral can be computed
only on the set $\{(x,\overline x)\in \re^n\times \re^n:|x-\overline{x}|<(x_n)_+/{2}\}$, which is contained in $(\re^n\setminus M)\times (\re^n\setminus M)$. Indeed
\begin{eqnarray*}&&\int\int_{\{|x-\overline{x}|\geq\frac{(x_n)_+}{2}\}}(x_n)_+^{1-2s}\frac{|\zeta(x)-\zeta(\overline{x})|^2}{|x-\overline{x}|^{n+1}}dx
d\overline{x}\\
&&\hspace{2em}\leq
2^{1-2s}\int\int_{\{|x-\overline{x}|\geq\frac{(x_n)_+}{2}\}}\frac{|\zeta(x)-\zeta(\overline{x})|^2}{|x-\overline{x}|^{n+2s}}dx
d\overline{x},\end{eqnarray*}
which can be absorbed in the first integral of the above bound.

Thus, if $0<s\leq 1/2$, we have
\begin{eqnarray*}&& \int_{\re^{n+1}_+}d_M(x,\lambda)^{1-2s}|\nabla \widetilde \zeta(x,\lambda)|^2dx d\lambda  \\
&& \hspace{0.5em} \leq C \int_{\re^n}\int_{\re^n}\frac{|\zeta(x)-\zeta(\overline x)|^2}{|x-\overline x|^{n+2s}}dx d\overline x 
+ C\int_{\re^n\setminus M}\int_{\re^n \setminus M}d_M(x)^{1-2s}\frac{|\zeta(x)-\zeta(\overline x)|^2}{|x-\overline x|^{n+1}}dx d\overline x.
\end{eqnarray*}

If $1/2<s<1$, set $$b=-a=2s-1>0.$$ In this case we use that $d_M(x,\lambda)\geq \max\{(x_n)_+,\lambda\}$, which leads to
$$d^a_M(x,\lambda)=1/d^b_M(x,\lambda)\leq 1/(\max\{(x_n)_+,\lambda\})^b.$$ In what follows we will use $d_M^a(x,\lambda)\leq 1/\lambda^b$ if $(x_n)_+=0$ and $d_M^a(x,\lambda)\leq 1/(x_n)_+^b$ if $(x_n)_+>0$. We have
\begin{eqnarray}
&&\int_{\re^{n+1}_+}d_M^a(x,\lambda)|\nabla \widetilde{\zeta}(x,\lambda)|^2dx
d\lambda \leq  \nonumber \\
&&\hspace{4em}\leq
C\int_0^{+\infty}d\lambda\int_{\{(x_n)_+=0\}}\int_{\{|x-\overline{x}|<\lambda\}}dx
d\overline{x}\frac{|\zeta(x)-\zeta(\overline{x})|^2}{\lambda^{n+2+b}}\nonumber \\
&&\hspace{5em}+C\int_0^{+\infty}d\lambda\int_{\{(x_n)_+>0\}}\int_{\{|x-\overline{x}|<\lambda\}}dx
d\overline{x}\frac{|\zeta(x)-\zeta(\overline{x})|^2}{(x_n)_+^b\lambda^{n+2}}\nonumber\\
&&\hspace{4em}\leq C\int_{\{(x_n)_+=0\}}dx\int_{\re^n}
d\overline{x}\frac{|\zeta(x)-\zeta(\overline{x})|^2}{|x-\overline{x}|^{n+2s}}\nonumber
\\
&&\hspace{5em}+ C\int_{\{(x_n)_+>0\}}dx\int_{\re^n}
d\overline{x}\frac{1}{(x_n)_+^b}\frac{|\zeta(x)-\zeta(\overline{x})|^2}{|x-\overline{x}|^{n+1}}
\nonumber \\
&&\hspace{4em}\leq C\int_{\{(x_n)_+=0\}} dx \int_{\{(\overline x_n)_+=0\}}
d\overline{x}\frac{|\zeta(x)-\zeta(\overline{x})|^2}{|x-\overline{x}|^{n+2s}}\label{term-a-s}
\\
&&\hspace{5em}+
C\int_{\{(x_n)_+>0\}}dx\int_{\re^n}
d\overline{x}\frac{1}{(x_n)_+^b}\frac{|\zeta(x)-\zeta(\overline{x})|^2}{|x-\overline{x}|^{n+1}}.\label{term-b-s}
\end{eqnarray}
Observe that the integral in \eqref{term-a-s} is computed only on the set
$\{(x,\overline{x})\in \re^n\times \re^n|(x_n)_+=0,
(\overline{x}_n)_+=0\}$. Indeed the set $L:=M\times
(\re^{n}\setminus M)=\{(x,\overline{x})\in \re^n\times
\re^n\:|\:(x_n)_+=0, (\overline{x}_n)_+>0\}\subseteq
\{(x,\overline{x})\in \re^n\times \re^n\:|\:(\overline{x}_n)_+ \leq
|x-\overline{x}|\}$. Then if $(x,\overline{x})\in L$
$$\frac{1}{|x-\overline{x}|^{n+1+b}}\leq
\frac{1}{(\overline{x}_n)^b_+}\cdot
\frac{1}{|x-\overline{x}|^{n+1}}$$ and hence we have that
$$\int_{\{(x_n)_+=0\}}dx\int_{\{(\overline{x}_n)_+>0\}}d\overline{x}\frac{|\zeta(x)-\zeta(\overline{x})|^2}{|x-\overline{x}|^{n+2s}}\leq
C\int_{\{(\overline{x}_n)_+>0\}}d\overline x \int_{\re^n}d{x}\frac{1}{(\overline{x}_n)_+^b}\frac{|\zeta(x)-\zeta(\overline{x})|^2}{|x-\overline{x}|^{n+1}},$$ which is equal to the integral in \eqref{term-b-s}.
This concludes the proof in the case of the half space.

\textbf{Case 2}: Let $\Omega \subset \re^{n+1}$ be a bounded domain with Lipschitz boundary $A=\partial \Omega$, and let $w\in C(\partial \Omega)$.

Let $\Gamma$ be the boundary (relative to $A=\partial \Omega$) of $M$ and let $\widetilde B_{r_i}=\widetilde B_{r_i}(p_i)\subset \re^{n+1}$ be the ball centered at $p_i\in \partial \Omega$ and of radius $r_i$. 
We set $A_{r_i}:=\widetilde B_{r_i}\cap \partial \Omega$. 
Let $Q_1$ denote the unite cube in $\re^n$.

Since $\partial \Omega$ is compact, we can consider a finite open covering of $\partial \Omega$,
$$ \bigcup_{i=1}^m A_{r_i}:=\bigcup_{i=1}^m (\widetilde B_{r_i}\cap \partial \Omega),$$
such that for every $i$ there exists a bilipschitz function $\varphi_i:\widetilde B_{r_i}\rightarrow Q_1\times (-1,1)$ which satisfies
\begin{equation}\label{phi1-s}
\varphi_i(\widetilde B_{r_i}\cap \Omega)= Q_1\times (0,1)\quad \mbox{and}\quad
 \varphi_i(A_{r_i})=Q_1\times \{0\}.
\end{equation} Moreover we may require that
\begin{itemize}
 \item if $\Gamma_i=A_{r_i} \cap \Gamma\neq \emptyset$, then
\begin{equation}\label{phi2-s}
 \varphi_i(\Gamma_i)=\{x \in Q_1:x_n=0\};
\end{equation}
\begin{equation}\label{phi3-s}
 \varphi_i(M\cap A_{r_i})=\{x \in Q_1:x_n<0\}=Q_1^-;
\end{equation}

\item if $\overline A_{r_i} \cap M= \emptyset$, then
$$r_i=\frac{1}{3}d_{M}(p_i),$$
where $p_i$ and $r_i$ are respectively the center and the radius of the ball $\widetilde B_{r_i}$.
\end{itemize}

Observe that the number $m$ of sets $A_{r_i}$ which cover $\partial \Omega$, and the Lipschitz constants of $\varphi_i$, depend only on $\partial \Omega$ and $\Gamma$.

We consider a partition of unity $\{\alpha_i\}_{i=1,...,m}$ relative to the covering $\{\widetilde{B}_{r_i}\}_{i=1,...,m}$, where $\alpha_i \in C^\infty_c(\widetilde{B}_{r_i})$ and $\sum_{i=1}^m \alpha_i=1$ on $\partial \Omega$.

If $w$ is a function defined on $\partial \Omega$, we write
$$w=\sum_{i=1}^m w\alpha_i=\sum_{i=1}^m w_i.$$
Using the bilipschitz map $\varphi_i$, we define
$$\zeta_i(y):=w_i(\varphi_i^{-1}(y,0))\quad \mbox{for every}\;\;y \in Q_1.$$
Then $\zeta_i$ has compact support in $Q_1$ and we extend it by $0$ outside $Q_1$ in all $\re^n$.

Next, we consider $\widetilde \zeta_i$, the extension of $\zeta_i$ in $\re^{n+1}_+$ defined by the convolution in \eqref{zeta-s}.
Since $\alpha_i\in C_c^{\infty}(\widetilde B_{r_i})$, there exists a function $\beta_i\in C_c^{\infty}(\widetilde B_{r_i})$ such that $\beta_i\equiv 1$ in the support of $\alpha_i$. Thus $\widetilde w_i:=\beta_i(\widetilde \zeta_ i \circ \varphi_i)$, extended by zero outside of $\widetilde B_{r_i}$, is well defined as a function in $\overline \Omega$ and agrees with $w\alpha_i=\beta_i w \alpha_i$ on $\partial \Omega$. We now define
$$\widetilde{w}=\sum_{i=1}^m \widetilde w_i=\sum_{i=1}^m\beta_i(\widetilde{\zeta}_i\circ \varphi_i) \quad
\mbox{in}\:\: \overline \Omega,$$
which agrees with $w$ on $\partial \Omega$.

Observe that, since $\varphi_i$ is a bilipschitz map and $\alpha_i,\:\beta_i \in C_c^{\infty}(\widetilde B_{r_i})$ for every
$i=1,...,m$, we have
$$|\nabla \widetilde{w}_i|\leq C\left\{|\nabla \beta_i|
|\widetilde{\zeta_i}\circ\varphi_i|+|\beta_i||(\nabla \widetilde{\zeta_i})\circ\varphi_i|\right\},$$
and thus
\begin{eqnarray*}\int_{\widetilde B_{r_i}\cap \Omega}
d^{1-2s}_M(z)|\nabla \widetilde{w}_i|^2 dz &\leq& C
\int_{\widetilde B_{r_i}\cap \Omega}
d^{1-2s}_M(z)|\widetilde \zeta_i \circ \varphi_i|^2 dz\\
 && \hspace{1em}+ C \int_{\widetilde B_{r_i}\cap \Omega}d^{1-2s}_M(z)|(\nabla
\widetilde{\zeta}_i)\circ \varphi_i|^2dz.\end{eqnarray*}
Observe that when $0<s \leq 1/2$, we have that $d^{1-2s}_M$ is uniformly bounded in $\Omega$, and thus using \eqref{L2-s} we get
\begin{eqnarray*}&& \int_{\widetilde B_{r_i}\cap \Omega}
d^{1-2s}_M(z)|\widetilde \zeta_i \circ \varphi_i|^2 dz\leq C\int_{\widetilde B_{r_i}\cap \Omega}
|\widetilde \zeta_i \circ \varphi_i|^2 dz\\
&&\hspace{1em}\leq C||\zeta||^2_{L^2(\re^n)}\leq C||w||^2_{L^2(\partial \Omega)}.\end{eqnarray*} 
On the other hand, when $1/2<s<1$, we use \eqref{L2w-s}, to obtain
\begin{eqnarray*}&&\int_{\widetilde B_{r_i}\cap \Omega}
d^{1-2s}_M(z)|\widetilde \zeta_i \circ \varphi_i|^2 dz\leq C \int_0^1\int_{Q_1}\lambda^{1-2s}|\widetilde \zeta_i(x,\lambda)|^2dx d\lambda \\
&&\hspace{1em} \leq C||\zeta||^2_{L^2(\re^n)}\leq C||w||^2_{L^2(\partial \Omega)}.\end{eqnarray*}
Thus
\begin{equation}\label{w_i}
 \int_{\widetilde B_{r_i}\cap \Omega}
d^{1-2s}_M(z)|\nabla \widetilde{w}_i|^2 dz \leq C||w||^2_{L^2(\partial \Omega)}+C \int_{\widetilde B_{r_i}\cap \Omega}d^{1-2s}_M(z)|(\nabla
\widetilde{\zeta}_i)\circ \varphi_i|^2dz.
\end{equation}
Using this bound we can prove the following \\
\textit{Claim}: \eqref{ext-s} holds with $\widetilde w$ and $w$ replaced by $\widetilde w_i$ and $w_i$, which have compact support in $\widetilde B_{r_i}\cap \overline\Omega$ and $A_{r_i}$ respectively.

It is enough to prove the claim to conclude the proof. Indeed, note first that
$$\int_{\Omega}d_{M}(z)^{1-2s}|\nabla \widetilde w|^2 dz \leq C\sum_{i=1}^m\int_{\widetilde B_{r_i}\cap \Omega}
d_{M}(z)^{1-2s}|\nabla \widetilde w_i|^2 dz.$$
Moreover, for every $i=1,...,m$,
\begin{eqnarray}\label{*-s}&&\int\int_{B_{\rm frac}}\frac{|w_i(z)-w_i(\overline{z})|^2}
{|z-\overline{z}|^{n+2s}}d\mathcal H^n({z}) d\mathcal H^n({\overline{z}}) \nonumber \\
&&\hspace{1em}\leq C||w||^2_{L^2(\partial
\Omega)}+C\int\int_{B_{\rm frac}}\frac{|w(z)-w(\overline{z})|^2}
{|z-\overline{z}|^{n+2s}}d\mathcal H^n({z}) d\mathcal H^n({\overline{z}})\nonumber.\end{eqnarray} 
Indeed,
\begin{eqnarray*}
&&\int \int_{B_{\rm frac}}\frac{|(w\alpha_i)(z)-(w\alpha_i)(\overline{z})|^2}
{|z-\overline{z}|^{n+2s}}d\mathcal H^n({z}) d\mathcal H^n({\overline{z}} ) \\
&& \hspace{1em}=\int\int_{B_{\rm frac}}\frac{|(w\alpha_i)(z)-w(\overline z)\alpha_i(z)+
w(\overline z)\alpha_i(z)-(w\alpha_i)(\overline{z})|^2}
{|z-\overline{z}|^{n+2s}}d\mathcal H^n({z}) d\mathcal H^n({\overline{z}}) \\
&& \hspace{1em}\leq  2\int \int_{B_{\rm frac}}\frac{|\alpha_i(z)-\alpha_i(\overline{z})|^2|w(\overline z)|^2}
{|z-\overline{z}|^{n+2s}}d\mathcal H^n({z}) d\mathcal H^n({\overline{z}}) \\
&& \hspace{2em}+ 2\int\int_{B_{\rm frac}}\frac{|w(z)-w(\overline{z})|^2|\alpha_i(z)|^2}
{|z-\overline{z}|^{n+2s}}d\mathcal H^n({z}) d\mathcal H^n({\overline{z}}) \\
&& \hspace{1em}\leq  C||w||^2_{L^2(\partial \Omega)}+C
\int \int_{B_{\rm frac}}\frac{|w(z)-w(\overline{z})|^2}
{|z-\overline{z}|^{n+2s}}d\mathcal H^n({z})
d\mathcal H^n({\overline{z}}),\end{eqnarray*}
where $C$ denotes
different positive constants depending on $\Omega$ and $s$. To get the bound $C||w||^2_{L^2(\partial \Omega)}$ for the first term,
we have used spherical coordinates centered at $\overline z$ and that $\alpha_i$ is Lipschitz.

Arguing in the same way, we deduce
\begin{eqnarray*}&&\int\int_{B_{\rm weig}}d_M(z)^{1-2s}\frac{|w_i(z)-w_i(\overline{z})|^2}
{|z-\overline{z}|^{n+1}}d\mathcal H^n({z}) d\mathcal H^n({\overline{z}})\\
&&\hspace{2em}\leq C||w||^2_{L^2(\partial
\Omega)}+C\int\int_{B_{\rm weig}}d_M(z)^{1-2s}\frac{|w(z)-w(\overline{z})|^2}
{|z-\overline{z}|^{n+1}}d\mathcal H^n({z}) d\mathcal H^n({\overline{z}}).\end{eqnarray*}
Indeed, using spherical coordinates centered at $\overline z$, that $\alpha_i$ is Lipschitz, and the definition \eqref{Bweig} of $B_{weig}$, 
we deduce (after flattening the boundary)
\begin{eqnarray*}
&&\int \int_{B_{\rm weig}}d_M(z)^{1-2s}\frac{|(w\alpha_i)(z)-(w\alpha_i)(\overline{z})|^2}
{|z-\overline{z}|^{n+1}}d\mathcal H^n({z}) d\mathcal H^n({\overline{z}} ) \\
&& \hspace{1em}\leq  2\int \int_{B_{\rm weig}}d_M(z)^{1-2s}\frac{|\alpha_i(z)-\alpha_i(\overline{z})|^2|w(\overline z)|^2}
{|z-\overline{z}|^{n+1}}d\mathcal H^n({z}) d\mathcal H^n({\overline{z}}) \\
&& \hspace{2em}+ 2\int\int_{B_{\rm weig}}d_M(z)^{1-2s}\frac{|w(z)-w(\overline{z})|^2|\alpha_i(z)|^2}
{|z-\overline{z}|^{n+1}}d\mathcal H^n({z}) d\mathcal H^n({\overline{z}}) \\
&& \hspace{1em}\leq  C||w||^2_{L^2(\partial \Omega)}+C
\int \int_{B_{\rm weig}}d_M(z)^{1-2s}\frac{|w(z)-w(\overline{z})|^2}
{|z-\overline{z}|^{n+1}}d\mathcal H^n({z})
d\mathcal H^n({\overline{z}}).\end{eqnarray*}

Next, we prove the claim.

Observe that we have three different cases, depending on the relative positions between the sets $A_{r_i}$ and $M$.

Case a). First, consider the case $\Gamma_i=A_{r_i}\cap \Gamma \neq \emptyset$.
By \eqref{w_i}, we have that
\begin{eqnarray}\label{local-s}
&&\int_{\widetilde B_{r_i}\cap \Omega}d_M(z)^{1-2s}|\nabla \widetilde w_i|^2 dz \nonumber\\
&&\hspace{2em}\leq C||w||^2_{L^2(\partial \Omega)}+C \int_{\re_+^{n+1}}((x_n)_+ +\lambda)^{1-2s}|\nabla \widetilde \zeta_i|^2 dx d\lambda.
\end{eqnarray}

Then, using the result in case 1, applied to $\widetilde \zeta_i$, we get
\begin{eqnarray*}
&&\int_{\widetilde B_{r_i}\cap \Omega}d_M(z)^{1-2s}|\nabla \widetilde w_i|^2 dz \\
&&\hspace{2em}\leq C||w||^2_{L^2(\partial \Omega)}+  C \int_{\re^{n+1}_+}((x_n)_+ 
+\lambda)^{1-2s}|\nabla\widetilde \zeta_i|^2 dx d\lambda\\
&& \hspace{2em} \leq  C||w||^2_{L^2(\partial \Omega)}+
C\int\int_{B_{\rm frac}}\frac{|\zeta_i(x)-\zeta_i(\overline
x)|^2}{|x-\overline x|^{n+2s}}dx d\overline
x\\
&& \hspace{3em} +C \int\int_{B_{\rm weig}}(x_n)_+^{1-2s}\frac{|\zeta_i(x)-\zeta_i(\overline
x)|^2}{|x-\overline x|^{n+1}}dx d\overline x,
\end{eqnarray*}
where $B_{\rm frac}$ and $B_{\rm weig}$ are defined as in
\eqref{Bfrac} and \eqref{Bweig} with $A=\re^{n}$ and $M=\{(x',x_n)\in
\re^n:x_n<0\}$.

Using the bilipschitz map $\varphi_i^{-1}$, we have
\begin{eqnarray*}
&&\int_{\widetilde B_{r_i}\cap \Omega}d_M(z)^{1-2s}|\nabla \widetilde w_i|^2 dz\\
&& \hspace{2em} \leq  C||w||^2_{L^2(\partial \Omega)}+  C\int\int_{B_{\rm frac}}\frac{| w_i(z)- w_i(\overline z)|^2}{|z-\overline z|^{n+2s}}d\mathcal H^n(z) d\mathcal H^n({\overline z})\\
&& \hspace{3em}+C \int\int_{B_{\rm weig}}d_M(z)^{1-2s}\frac{| w_i(z)- w_i(\overline z)|^2}{|z-\overline z|^{n+1}}d\mathcal H^n(z) d\mathcal H^n({\overline z}),
\end{eqnarray*}
where now, $B_{\rm frac}$ and $B_{\rm weig}$ are defined as in \eqref{Bfrac} and \eqref{Bweig} with $A=\partial \Omega$.

Case b). Second, consider the case $A_{r_i} \subset M$. In this case, the claim follows exactly as in case a), with $(x_n)_+=0$ in \eqref{local-s}.

Case c). Finally, consider the case $A_{r_i} \subset \partial \Omega \setminus M$.

We recall that, by construction $A_{r_i}=\widetilde B_{r_i} \cap \partial \Omega$ where $\widetilde B_{r_i}$ is the ball centered at $p_i \in \partial \Omega \setminus M$ and of radius $\displaystyle r_i=\frac{1}{3}d_{M}(p_i).$

Thus, for every $z \in \widetilde B_{r_i}\cap \Omega$, we have that
$$\frac{2}{3}d_{M}(p_i)\leq d_{M}(z)\leq \frac{4}{3}d_{M}(p_i).$$

Then, for every $i=1,...,m$
$$\int_{\widetilde B_{r_i} \cap \Omega} d_{M}(z)^{1-2s}|\nabla \widetilde w_i|^2 dz \leq C d_{M}(p_i)^{1-2s}\int_{\widetilde B_{r_i} \cap \Omega}|\nabla \widetilde w_i|^2 dz.$$
Observe that the integral on the right-hand side does not contain weights. Moreover, we recall that the extension $\widetilde w_i$ is defined as for the case $s=1/2$. Thus, applying the extension result given in \cite{C-Cinti1} for $s=1/2$, we get
\begin{eqnarray*}&&\int_{\widetilde B_{r_i} \cap \Omega} d_{M}(z)^{1-2s}|\nabla \widetilde w_i|^2 dz \\
&& \hspace{2em} \leq C||w||^2_{L^2(\partial \Omega)}+ C d_M(p_i)^{1-2s} \int\int_{B_{\rm weig}}\frac{| w_i(z)- w_i(\overline z)|^2}{|z-\overline z|^{n+1}}d\mathcal H^n(z) d\mathcal H^n({\overline z}) \\
&& \hspace{2em} \leq  C||w||^2_{L^2(\partial \Omega)}+  C \int\int_{B_{\rm weig}}d_M(z)^{1-2s}\frac{| w_i(z)- w_i(\overline z)|^2}{|z-\overline z|^{n+1}}d\mathcal H^n(z) d\mathcal H^n({\overline z}),\end{eqnarray*} where $B_{\rm frac}$ and $B_{\rm weig}$ are defined as in \eqref{Bfrac} and \eqref{Bweig} with $A=\partial \Omega$.
This concludes the proof of the claim.
\end{proof}

We give now the proof of Theorem \ref{key-s}.

\begin{proof}[Proof of Theorem \ref{key-s}] Let $s\in(0,1).$ When $s=1/2$ the theorem was proved in our previous work \cite{C-Cinti1}, Theorem 1.5. 
Here we will prove it for $0<s<1/2$ and $1/2<s<1$. 
The proof differs in each of these two cases, they also differ from the one for $s=1/2$ in which there is a $|\log \varepsilon|$ appearing in the final bound.

\textbf{Step 1}. Suppose that $A=Q_1=\{x\in \re^n:|x|<1\}$ is the unit
cube in $\re^n$. 
 Let, as before, $(x',x_n)\in \re^n$. As in the proof of Theorem \ref{extension-s}, we consider $M=Q_1^-=\{x \in Q_1 : x_n<0\}$ and
$\Gamma=\{x \in Q_1 : x_n=0\}$.
In the following computations $C$ will denote different positive constants which depend on $n$, $s$, and $c_s$. 

\textbf{Case $0<s<1/2$}. By hypothesis $|w(x)|\leq c_s$ and by \eqref{grad-1A} we have
\begin{equation}\label{bound-s1}
\displaystyle |D w(x)|\leq \begin{cases}
\displaystyle \frac{c_s}{\varepsilon}\left(\frac{d_{\Gamma}(x)}{\varepsilon}\right)^{2s-1}=\frac{c_s}{\varepsilon}\left(\frac{|x_n|}{\varepsilon}\right)^{2s-1}& \mbox{in}\:\:A\cap
|x_n|\leq \varepsilon\}\\
\displaystyle \frac{c_s}{d_{\Gamma}(x)}=\frac{c_s}{|x_n|} & \mbox{in}\:\:A\cap\{|x_n| > \varepsilon\}.\end{cases}\end{equation}

Let $Q_1^+=\{x\in Q_1 : x_n>0\}$. Following \eqref{Bfrac} and \eqref{Bweig} we must consider the quantity
\begin{eqnarray*}
&&I:=\int_{Q_1}\int_{Q_1}\frac{|w(x)-w(\overline{x})|^2}{|x-\overline{x}|^{n+2s}}dx
d\overline{x}\\
&&\hspace{2em}+\int_{Q_1^+}\int_{Q_1^+}(x_n)_+^{1-2s}\frac{|w(x)-w(\overline{x})|^2}
{|x-\overline{x}|^{n+1}}dx d\overline{x}.\end{eqnarray*}
Since hypothesis \eqref{bound-s1} is symmetric in $x_n$ and $-x_n$, we simply bound $I$ by
\begin{eqnarray}\label{I}
I\leq \int_{Q_1}\int_{Q_1}\frac{|w(x)-w(\overline{x})|^2}{|x-\overline{x}|^{n+2s}}dx
d\overline{x}+\int_{Q_1}\int_{Q_1}|x_n|^{1-2s}\frac{|w(x)-w(\overline{x})|^2}
{|x-\overline{x}|^{n+1}}dx d\overline{x}.\end{eqnarray}
Observe
that, in the set $\{|x-\overline{x}|<|x_n|/2\}$, we have
$$\frac{|w(x)-w(\overline{x})|^2}{|x-\overline{x}|^{n+2s}}\leq 2^{2s-1} |x_n|^{1-2s}\frac{|w(x)-w(\overline{x})|^2}
{|x-\overline{x}|^{n+1}},$$ while the reverse inequality holds in  $\{|x-\overline{x}|\geq |x_n|/2\}$.
We deduce that
\begin{equation}\label{dis1-s}
\begin{array}{ll}\displaystyle  I \leq C\int_{Q_1}\int_{Q_1\cap \{\overline x:|x-\overline x|>|x_n|/2\}}\frac{|w(x)-w(\overline{x})|^2}
{|x-\overline{x}|^{n+2s}}dx d\overline{x}\\
\displaystyle \hspace{1em} +C\int_{Q_1}\int_{Q_1\cap \{\overline x:|x-\overline x|<|x_n|/2\}}|x_n|^{1-2s}\frac{|w(x)-w(\overline{x})|^2}
{|x-\overline{x}|^{n+1}}dx d\overline{x}=:I_1+I_2.\end{array}
\end{equation}

We bound $I_1$ using the $L^{\infty}$ estimate for $w$ and spherical coordinates centered at $x$:
\begin{eqnarray}\label{dis2-s}
&&\int_{Q_1}\int_{Q_1\cap\{\overline{x}:|x-\overline{x}|>|x_n|/2\}}\frac{|w(x)-w(\overline{x})|^2}{|x-\overline{x}|^{n+2s}}dx
d\overline{x}
\leq C\int_{Q_1}dx\int_{|x_n|/2}^{2\sqrt n}dr \frac{1}{r^{2s+1}}\nonumber \\
&&\hspace{3em}\leq
C\int_{-1}^1\frac{1}{|x_n|^{2s}}dx_n\leq C.\end{eqnarray}
Next, we consider $I_2$. 
By symmetry between $x_n$ and $-x_n$ we can suppose $x\in Q_1^+$. Using the gradient bounds \eqref{bound-s1} 
for $w$ and spherical coordinates centered at $x$,  we get
\begin{eqnarray*}
&&I_2\leq C \int_{Q_1^+}\int_{Q_1\cap\{\overline{x}:|x-\overline{x}|<x_n/2\}}x_n^{1-2s}\frac{|w(x)-w(\overline{x})|^2}
{|x-\overline{x}|^{n+1}}dx d\overline{x}\nonumber \\
&&\hspace{1em}\leq
C\int_0^{\varepsilon/2}dx_n x_n^{1-2s}\int_0^{x_n/2}dr\frac{1}{\varepsilon^2}\left(\frac{|y_n(x,\overline x)|}{\varepsilon}\right)^{2(2s-1)}\nonumber \\
&& \hspace{1.5em}+
C\int_{\varepsilon/2}^1 dx_n
x_n^{1-2s}\int_0^{x_n/2}dr\frac{1}{y_n^2(x,\overline x)},
\end{eqnarray*}
where $y(x,\overline x)\in B_{x_n/2}(x)$ is a point of the segment joining $x$ and $\overline x$. In the first integral in the last bound, we have used that $y_n(x,\overline x) \leq x_n+x_n/2 \leq 2x_n \leq \varepsilon$. In the second integral we have used that, in case $y_n(x,\overline x) \leq \varepsilon$, $\frac{C}{\varepsilon}\left(\frac{y_n(x,\overline x)}{\varepsilon}\right)^{2s-1}\leq \frac{C}{y_n(x,\overline x)}$ in \eqref{bound-s1}.

Since $y(x,\overline x)\in B_{x_n/2}(x)$, we have that $y_n(x,\overline x)\geq x_n/2$ and thus we deduce that
\begin{eqnarray}\label{dis3-s}
&&\int_{Q_1^+}\int_{\{\overline{x}\in
Q_1:|x-\overline{x}|<x_n/2\}} x_n^{1-2s}\frac{|w(x)-w(\overline{x})|^2}
{|x-\overline{x}|^{n+1}}dx d\overline{x}\nonumber \\
&&\hspace{1em}\leq
C\int_0^{\varepsilon/2}dx_n x_n^{1-2s}\int_0^{x_n/2}dr\frac{1}{\varepsilon^2}\left(\frac{x_n}{\varepsilon}\right)^{2(2s-1)}+
C\int_{\varepsilon/2}^1 dx_n
x_n^{1-2s}\int_0^{x_n/2}dr\frac{1}{x_n^2}\nonumber \\
&&\hspace{1em}\leq C\frac{1}{\varepsilon^{4s}}\int_0^{\varepsilon/2}x_n^{2s}dx_n+C
\int_{\varepsilon/2}^1 x_n^{-2s}dx_n\nonumber\\
&&\hspace{1em}\leq C\varepsilon^{1-2s}+C\int_{\varepsilon/2}^1 x_n^{-2s}dx_n\leq C\int_\varepsilon^1 x_n^{-2s}dx_n.
\end{eqnarray}
Using \eqref{dis1-s}, \eqref{dis2-s}, and \eqref{dis3-s}, we conclude that $$I \leq C\int_\varepsilon^1x_n^{-2s}dx_n.$$

\textbf{Case} $1/2<s<1$.
By \eqref{Bfrac} and \eqref{Bweig} we must  
consider now the quantity
\begin{equation*}
\widetilde I:=\int_{Q_1^-}\int_{Q_1^-}\frac{|w(x)-w(\overline{x})|^2}
{|x-\overline{x}|^{n+2s}}dx
d\overline{x}
+\int_{Q_1^+}\int_{Q_1} \;x_n^{1-2s}\frac{|w(x)-w(\overline{x})|^2}
{|x-\overline{x}|^{n+1}}dx d\overline{x}.\end{equation*} 
We recall that in this case, by \eqref{grad-2M} we have
\begin{equation}\label{grad_bound-s}
|D w(x)|\leq\begin{cases} 
\displaystyle \frac{c_s}{\varepsilon} & \mbox{in}\:  A\cap \{|x_n|\leq \varepsilon\}  \\
\displaystyle \frac{c_s}{|x_n|}&\mbox{in}\: A\cap \{|x_n|> \varepsilon\} . 
\end{cases}\end{equation}

We have that
\begin{eqnarray*}
 \widetilde I&\leq& \int_{Q_1}\int_{Q_1}\frac{|w(x)-w(\overline{x})|^2}
{|x-\overline{x}|^{n+2s}}dx
d\overline{x}\\
&&\hspace{0.5em}+\int_{Q_1}\int_{Q_1} \;|x_n|^{1-2s}\frac{|w(x)-w(\overline{x})|^2}
{|x-\overline{x}|^{n+1}}dx d\overline{x}\\
&\leq &  C \int_{Q_1}\int_{Q_1\cap\{\overline x:|x-\overline x|\leq |x_n|/2\}}\frac{|w(x)-w(\overline{x})|^2}
{|x-\overline{x}|^{n+2s}}dx d\overline{x} \\
&&\hspace{0.5em}+  C\int_{Q_1}\int_{Q_1\cap\{\overline x:|x-\overline x|> |x_n|/2\}} |x_n|^{1-2s}\frac{|w(x)-w(\overline{x})|^2}
{|x-\overline{x}|^{n+1}}dx
d\overline{x}\\
&\leq &  C \int_{Q_1}\int_{Q_1\cap\{\overline x:|x-\overline x|\leq |x_n|/2\}}\frac{|w(x)-w(\overline{x})|^2}
{|x-\overline{x}|^{n+2s}}dx d\overline{x} \\
&&\hspace{0.5em} + C\int_{Q_1}\int_{Q_1\cap\{\overline x: |x_n|/2< |x-\overline x|\leq \max\{\varepsilon/2,|x_n|/2\}\}} |x_n|^{1-2s}\frac{|w(x)-w(\overline{x})|^2}
{|x-\overline{x}|^{n+1}}dx
d\overline{x}\\
&&\hspace{0.5em} + C\int_{Q_1}\int_{Q_1\cap\{\overline x:|x-\overline x|\geq \max\{\varepsilon/2,|x_n|/2\}\}} |x_n|^{1-2s}\frac{|w(x)-w(\overline{x})|^2}
{|x-\overline{x}|^{n+1}}dx
d\overline{x}\\
&=&\widetilde I_1+\widetilde I_2+\widetilde I_3.
\end{eqnarray*}

We first bound $\widetilde I_1$.   We have
$$\widetilde I_1\leq C \int_{Q_1}\int_{Q_1\cap\{\overline x:|x-\overline x|\leq |x_n|/2\}}\frac{|Dw(y(x,\overline x))|^2}{|x-\overline x|^{n-2+2s}}dx d\overline x,$$
where $y(x,\overline x)\in B_{|x_n|/2}(x)$ is a point of the segment joining $x$ and $\overline x$. Now, 
the gradient bound \eqref{grad_bound-s} reads $|Dw(y)|\leq c_s\min\{\varepsilon^{-1},|y_n|^{-1}\}$ for a.e. $y \in Q_1$.  
Since $y(x,\overline x)\in B_{|x_n|/2}(x)$, we have $|y_n(x,\overline x)|\geq |x_n|/2$ and
$|Dw(y(x,\overline x))|\leq c_s\min\{\varepsilon^{-1},|y_n(x,\overline x)|^{-1}\}\leq c_s\min\{\varepsilon^{-1},2|x_n|^{-1}\}$. 
Using spherical coordinates centered at $x$, we get
\begin{eqnarray*}
 \widetilde I_1 &\leq& C\int_{Q_1}dx\int_0^{|x_n|/2}dr r^{1-2s}\min\left\{\frac{1}{\varepsilon^2},\frac{1}{|x_n|^2}\right\}\\
&\leq &C \int_{-1}^{1}dx_n \min\left\{\frac{1}{\varepsilon^2},\frac{1}{|x_n|^2}\right\} |x_n|^{2-2s}\\
&\leq& C \int_0^{\varepsilon}\frac{1}{\varepsilon^2}x_n^{2-2s} dx_n + C\int_{\varepsilon}^1 x_n^{-2s} dx_n\leq C\int_{\varepsilon}^1 x_n^{-2s} dx_n.
\end{eqnarray*}

Consider now $\widetilde I_2$. Here $|x_n|< \varepsilon$ (if not $\{|x_n|/2< \max\{|x_n|/2,\varepsilon/2\}\}=\emptyset$). Using that $\widetilde I_2$ is symmetric in $x_n$ and $-x_n$, the gradient bound \eqref{grad_bound-s} and spherical coordinates centered at $x$ as for $\widetilde I_1$, we get
$$\widetilde I_2\leq C\int_0^{\varepsilon}dx_n x_n^{1-2s}\int_{x_n/2}^{\varepsilon/2} dr \frac{1}{\varepsilon^2}\leq  C\varepsilon^{1-2s}.$$

Finally, using that $|w|\leq c_s$ in $Q_1$ and spherical coordinates centered at $x$, we get the following bound for $\widetilde I_3$:
\begin{eqnarray*}
\widetilde I_3 &\leq& C\int_{Q_1}dx_n |x_n|^{1-2s}\int_{\max\{|x_n|/2,\varepsilon/2\}}^{2\sqrt n} dr \frac{1}{r^2}\\
&\leq& C\int_{-1}^1 dx_n |x_n|^{1-2s}\min\left\{\frac{1}{|x_n|},\frac{1}{\varepsilon}\right\}\\
&\leq& C\int_0^{\varepsilon}\frac{1}{\varepsilon}x_n^{1-2s}dx_n+C\int_{\varepsilon}^1 x_n^{-2s}dx_n\\
&\leq& C\varepsilon^{1-2s}+C\int_{\varepsilon}^1 x_n^{-2s}dx_n\leq C\int_{\varepsilon}^1 x_n^{-2s}dx_n.
\end{eqnarray*}
We conclude that $$\widetilde I\leq C\int_{\varepsilon}^1 x_n^{-2s}dx_n.$$

\textbf{Step 2.} Suppose now that $A$ is a Lipschitz subset of
$\re^n$ or $A=\partial \Omega$, where $\Omega$ is an open bounded
subset of $\re^{n+1}$ with Lipschitz boundary.

We consider a finite open covering
$\{A_{r_i/2}\}_{i=1,...,m}=\{B_{r_i/2}\cap A\}_{i=1,...,m}$,
where now $B_{r_i/2}$ is the ball centered at $p_i$ (as in the proof of Theorem \ref{extension-s}, case 2) but of radius $r_i/2$.
Here, for sake of simplicity, $B_{r_i}$ denotes both the ball in
$\re^n$ or $\re^{n+1}$.

Let $\overline \Gamma$ be the closure of $\Gamma$ in $\re^n$ or $\re^{n+1}$.
Only in the case $A\subset \re^n$, it may happen that $\overline \Gamma \setminus \Gamma \neq \emptyset$. In such case, for $p_i\in \overline \Gamma \setminus \Gamma$, there exists a radius $r_i$ and a bilipschitz diffeomorphism $\varphi_i: B_{r_i}(p_i)\rightarrow (-3,1)\times (-1,1)^{n-1}$ such that $\varphi_{i}(p_i)=(-1,0,...,0)$ and $\varphi_i$ satisfies properties \eqref{phi1-s}, \eqref{phi2-s}, and \eqref{phi3-s}.
We set $\varepsilon_0=\min\{r_i/2,1/2\}$.

If $z$ and $\overline z$ are two points belonging to $A$ such that
$|z-\overline z|<\varepsilon_0$, then there exists a set
$A_{r_i}=B_{r_i}\cap A$ such that both $z$ and $\overline z$
belong to $A_{r_i}$. Hence
$$\{(z,\overline z)\in A \times A:
|z-\overline z|<\varepsilon_0\} \subset \displaystyle
\bigcup_{i=1}^m A_{r_i}\times A_{r_i}.$$ Let $L>1$ be a bound
for the Lipschitz constants of all functions
$\varphi_{1},...,\varphi_{m},\varphi_{1}^{-1},...,\varphi_{m}^{-1}.$
Let us first treat the case $0< \varepsilon \leq 1/(2L)$.

We write
\begin{eqnarray*}
&&\int\int_{B_{\rm frac}} \frac{|w(z)-w(\overline z)|^2}{|z-\overline z|^{n+2s}}d\mathcal H^n(z) d\mathcal H^n({\overline z})\\
&&\hspace {1em}+\int\int_{B_{\rm weig}} d_M(z)^{1-2s}\frac{|w(z)-w(\overline z)|^2}{|z-\overline z|^{n+1}}d\mathcal H^n(z) d\mathcal H^n({\overline z})\\
&&\hspace {2em}=\int\int_{B_{\rm frac}\cap\{\overline z:|z-\overline z|<\varepsilon_0\}} \frac{|w(z)-w(\overline z)|^2}{|z-\overline z|^{n+2s}}d\mathcal H^n(z) d\mathcal H^n({\overline z})\\
&&\hspace {3em}+\int\int_{B_{\rm frac}\cap\{\overline z:|z-\overline z|>\varepsilon_0\}} \frac{|w(z)-w(\overline z)|^2}{|z-\overline z|^{n+2s}}d\mathcal H^n(z) d\mathcal H^n({\overline z})\\
&&\hspace {3em}+ \int\int_{B_{\rm weig}\cap\{\overline z:|z-\overline z|<\varepsilon_0\}} d_M(z)^{1-2s}\frac{|w(z)-w(\overline z)|^2}{|z-\overline z|^{n+1}}d\mathcal H^n(z) d\mathcal H^n({\overline z})\\
&&\hspace {3em} +\int\int_{B_{\rm weig}\cap\{\overline z:|z-\overline
z|>\varepsilon_0\}} d_M(z)^{1-2s}\frac{|w(z)-w(\overline
z)|^2}{|z-\overline z|^{n+1}}d\mathcal H^n(z) d\mathcal H^n({\overline z}).
\end{eqnarray*}
Since $w$ is bounded and $\int_{B_{\rm weig}}d_M(z)^{1-2s}dz\leq C$ for every
$0<s<1$, using spherical coordinates centered at $z$ as before, we have that
\begin{eqnarray*}&&\int\int_{B_{\rm frac}\cap\{\overline z:|z-\overline z|>\varepsilon_0\}} \frac{|w(z)-w(\overline z)|^2}{|z-\overline z|^{n+2s}}d\mathcal H^n(z) d\mathcal H^n({\overline z})\\
&&\hspace{1em} +\int\int_{B_{\rm weig}\cap\{\overline z:|z-\overline z|>\varepsilon_0\}} d_M(z)^{1-2s}\frac{|w(z)-
w(\overline z)|^2}{|z-\overline z|^{n+1}}d\mathcal H^n(z) d\mathcal H^n({\overline z})\leq C.
\end{eqnarray*}

On the other hand, by the previous consideration,
\begin{eqnarray*}
&&\int\int_{B_{\rm frac}\cap\{\overline z:|z-\overline z|<\varepsilon_0\}} \frac{|w(z)-w(\overline z)|^2}{|z-\overline z|^{n+2s}}d\mathcal H^n(z) d\mathcal H^n({\overline z})\\
&&\hspace{3em}+\int\int_{B_{\rm weig}\cap\{\overline z:|z-\overline z|<\varepsilon_0\}} d_M(z)^{1-2s}\frac{|w(z)-w(\overline z)|^2}{|z-\overline z|^{n+1}}d\mathcal H^n(z) d\mathcal H^n({\overline z})\\
&&\hspace{2em} \leq  \sum_{i=1}^m \int\int_{B_{\rm frac} \cap (A_{r_i}\times A_{r_i})}\frac{|w(z)-w(\overline z)|^2}{|z-\overline z|^{n+2s}}d\mathcal H^n(z) d\mathcal H^n({\overline z})\\
&&\hspace{3em} +\sum_{i=1}^m\int\int_{B_{\rm weig} \cap (A_{r_i}\times
A_{r_i})}d_M(z)^{1-2s}\frac{|w(z)-w(\overline z)|^2}{|z-\overline
z|^{n+1}}d\mathcal H^n(z) d\mathcal H^n({\overline z}).
\end{eqnarray*}

If $A_{r_i}\cap \Gamma \neq \emptyset$ or $A_{r_i}\subset M$ then, by the construction
of the open covering $\{A_{r_i}\}$, there exists a bilipschitz map
$\varphi_i:B_{r_i}\rightarrow Q_1 \times (-1,1)$ such that
$\varphi_i(A_{r_i})=Q_1$. Moreover, if $A_{r_i}\cap \Gamma \neq \emptyset$, we have also $\varphi_i(A_{r_i}\cap M)=\{x\in
Q_1:x_n<0\}$. We use the bilipschitz map $\varphi_i$ to flatten
the sets $B_{\rm frac} \cap (A_{r_i}\times A_{r_i})$ and $B_{\rm weig} \cap
(A_{r_i}\times A_{r_i})$, and we set $v_i=w\circ \varphi_i^{-1}$.
Given $x \in Q_1$, let $y=\varphi_i^{-1}(x) \in A_{r_i}$.
Recalling that $L>1$ is a bound for the Lipschitz constants of all functions
$\varphi_{1},...,\varphi_{m},\varphi_{1}^{-1},...,\varphi_{m}^{-1}$, we
have that $(1/L)d_{\Gamma}(y)\leq |x_n|\leq Ld_{\Gamma}(y)$ and $|Dv_i(x)|\leq L|Dw(y)|$. Therefore the gradient estimates \eqref{grad-1A} and \eqref{grad-2M} lead to the following bounds for $|Dv_i|$:

\begin{itemize}
\item for every $s\in (0,1/2]$,
\begin{eqnarray*}
|Dv_i(x)|\leq L|Dw(y)|&\leq& \begin{cases}
\displaystyle L\frac{c_s}{\varepsilon}\left(\frac{d_\Gamma(y)}{\varepsilon}\right)^{2s-1} & \mbox{if} \; y\in A\;\;\mbox{and}\;\;d_\Gamma(y)\leq \varepsilon\\
\displaystyle L\frac{c_s}{d_\Gamma(y)}& \mbox{if} \; y\in A\;\;\mbox{and}\;\;d_\Gamma(y)>\varepsilon\end{cases}\\
&\leq& \begin{cases}
\displaystyle L\frac{c_s}{\varepsilon}\left(\frac{|x_n|}{L\varepsilon}\right)^{2s-1} & \mbox{if} \; x\in Q_1\;\;\mbox{and}\;\;|x_n|\leq L\varepsilon\\
\displaystyle L^2\frac{c_s}{|x_n|}& \mbox{if} \; x\in Q_1\;\;\mbox{and}\;\;|x_n|> L\varepsilon\end{cases}
\end{eqnarray*}
\item for every $s\in (1/2,1)$,
\begin{eqnarray*}
|Dv_i(x)|\leq L|Dw(y)|&\leq& \begin{cases}
\displaystyle L\frac{c_s}{\varepsilon} & \mbox{if} \; y\in A\;\;\mbox{and}\;\;d_\Gamma(y)\leq\varepsilon\\
\displaystyle L\frac{c_s}{d_\Gamma(y)}& \mbox{if} \; y\in A\;\;\mbox{and}\;\;d_\Gamma(y)>\varepsilon\end{cases}\\
&\leq& \begin{cases}
\displaystyle L\frac{c_s}{\varepsilon} & \mbox{if} \; x\in Q_1\;\;\mbox{and}\;\;|x_n|\leq L\varepsilon\\
\displaystyle L^2\frac{c_s}{|x_n|}& \mbox{if} \; x\in Q_1\;\;\mbox{and}\;\;|x_n|> L\varepsilon.\end{cases}
\end{eqnarray*}
\end{itemize}

Thus we can apply the result proven in Step 1, with $\varepsilon$
replaced by $\varepsilon L$ (note that we have $\varepsilon L \leq
1/2$, as in Step 1), to the function $v_i/(1+L^2)$. Using the Lipschitz property of $\varphi^{-1}$,
we restate the conclusion for $w$ and we get
\begin{eqnarray*}
&&\hspace{-0.4em}\int\int_{B_{\rm frac} \cap (A_{r_i}\times A_{r_i})}\frac{|w(z)-w(\overline z)|^2}{|z-\overline z|^{n+2s}}d\mathcal H^n(z) d\mathcal H^n({\overline z})\\
&&+\int\int_{B_{\rm weig} \cap (A_{r_i}\times A_{r_i})}d_M(z)^{1-2s}\frac{|w(z)-w(\overline z)|^2}{|z-\overline z|^{n+1}}d\mathcal H^n(z) d\mathcal H^n({\overline z})\leq C\int_{\varepsilon}^1\rho^{-2s}d\rho.
\end{eqnarray*}

Last, we consider the case $A_{r_i}\cap M= \emptyset$. We recall that, in this case $\displaystyle r_i=\frac{1}{3}d_M(p_i)$, 
where $r_i$ and $p_i$ are respectively the radius and the center of the ball $B_{r_i}$. Then, for every $z \in A_{r_i}$, 
we have that $d_\Gamma(z)\geq d_M(z)\geq r_i\geq \varepsilon_0 $ and thus we have that $\displaystyle |D w(z)|\leq  
\frac{c_s}{\varepsilon_0}.$

Using this gradient bound to have $|Dw(y(z,\overline z))|\leq C$, where $y(z,\overline z)$ is a point of 
the segment joining $z$ and $\overline z$, and using also spherical coordinates, we get
\begin{eqnarray*}
&&\int\int_{B_{\rm frac} \cap (A_{r_i}\times A_{r_i})}\frac{|w(z)-w(\overline z)|^2}
{|z-\overline z|^{n+2s}}d\mathcal H^n(z) d\mathcal H^n({\overline z})\\
&&\hspace{4em}+\int\int_{B_{\rm weig} \cap (A_{r_i}\times A_{r_i})}d_M(z)^{1-2s}\frac{|w(z)-w(\overline z)|^2}
{|z-\overline z|^{n+1}}d\mathcal H^n(z) d\mathcal H^n({\overline z})\\
&&\hspace{3em}\leq C \int_{A_{r_i}}d\mathcal H^n(z)
\int_0^{r_i}dr\,\frac{1}{r^{2s-1}}+C\int_{A_{r_i}}d\mathcal H^n(z)  d_M(z)^{1-2s}\int_0^{r_i}dr\leq C.
\end{eqnarray*} 
Summing over $i=1,...,m$, we conclude
the proof in case $\varepsilon \leq 1/(2L)$.

Finally given $\varepsilon \in (0,1/2)$ with $\varepsilon >
1/(2L)$, since \eqref{grad-1A} and \eqref{grad-2M} hold with such $\varepsilon$,
they also hold with $\varepsilon$ replaced by $1/(2L)$. By the
previous proof with $\varepsilon$ taken to be $1/(2L)$, the energy
is bounded by
$$C\int_{1/(2L)}^1\rho^{-2s}d\rho\leq C\leq C\int_{\varepsilon}^1\rho^{-2s}d\rho.$$
\end{proof}

\section{Energy estimate for global minimizers}

In this section we give the proof of Theorem \ref{energy-est-s},
which is based on a comparison argument. Let $v$ be a global
minimizer of $\eqref{eq2-s}$. The proof consists of 3 steps:
\begin{enumerate}
\item [i)] solving a Dirichlet problem, we construct an appropriate bounded comparison function $\overline{w}$ which
takes the same values as $v$ on $\partial C_R \cap
\{\lambda>0\}$ and thus, by minimality of $v$,
$$\mathcal E_{s,C_R}(v)\leq \mathcal E_{s,C_R}(\overline{w});$$
\item [ii)] we apply the extension Theorem \ref{extension-s} in the
cylinder of radius $R$ and height $R$ to deduce
\begin{equation}\label{estimate_C1-s}
\begin{split}
&\int_{C_R}\lambda^{1-2s}|\nabla \overline{w}|^2dx d\lambda \leq C||w||^2_{L^2(\partial C_R)}+\\
&\hspace{6em}+C
\int\int_{B_{\rm frac}}\frac{|w(z)-w(\overline{z})|^2}{|z-\overline{z}|^{n+2s}}d\mathcal H^n(z)d\mathcal H^n({\overline{z}}) \\
&\hspace{6em}+
C \int\int_{B_{\rm weig}}
\lambda^{1-2s}\frac{|w(z)-w(\overline{z})|^2}{|z-\overline{z}|^{n+1}}d\mathcal H^n(z)d\mathcal H^n({\overline{z}}),\end{split}\end{equation}
where $z\in \partial C_R$, $w$ is the trace of $\overline{w}$ on
$\partial C_R$ and $B_{\rm frac}$ and $B_{\rm weig}$ are defined as in \eqref{Bfrac}
and \eqref{Bweig}, with $A=\partial C_R$ and $M=B_R \times
\{0\}.$

\item [iii)] we prove, rescaling and using Theorem \ref{key-s}, that the quantity in the right-hand side of
\eqref{estimate_C1-s} is bounded by $CR^{n-2s}\int_{1/R}^1\rho^{-2s}d\rho.$

\end{enumerate}

\begin{proof}[Proof of Theorem \ref{energy-est-s}]

Let $v$ be a bounded global minimizer of \eqref{eq2-s}. Let $u$ be
its trace on $\partial \re^{n+1}_+$. Let $\tau \in [\inf u, \sup u]$
be such that $G(\tau)=c_u$, where $c_u$ as defined as in \eqref{c_u-s}.

Throughout the proof, $C$ will denote positive constants depending
only on $n$, $s$, $||f||_{C^{1,\gamma}}$ and $||u||_{L^\infty(\re^n)}$. As
explained in \eqref{grad-s}, $v$ satisfies the following bounds:
\begin{equation}\label{grad-1}
|\nabla_x v(x,\lambda)|\leq c_s\:\:\mbox{for
every}\;\;x\in\re^n\;\;\mbox{and}\;\;\lambda\geq 0\end{equation}
\begin{equation}\label{grad-2}
|\nabla v(x,\lambda)|\leq \frac{c_s}{\lambda}\:\:
\mbox{for every}\;\;x\in\re^n\;\;\mbox{and}\;\;\lambda>0
\end{equation}
\begin{equation}\label{grad-3} |\lambda^{1-2s}\partial_{\lambda}v|\leq c_s \:\: \mbox{for
every}\;\;x\in\re^n\;\;\mbox{and}\;\;\lambda>0.\end{equation}
We recall (see Remark \ref{oss-grad}) that the constants $c_s$ in \eqref{grad-1} and \eqref{grad-3} are  not uniformly bounded for $s$ close to $0$ and $s$ close to $1$ respectively.

We estimate the energy $\mathcal{E}_{s,C_R}(v)$ of $v$ using a
comparison argument. We define a function $\overline
w=\overline{w}(x,\lambda)$ in $C_R$
 in the following way. First we define $\overline{w}(x,0)$
on the base of the cylinder to be equal to a smooth function
$g(x)$ which is identically equal to $\tau$ in $B_{R-1}$ and
$g(x)=v(x,0)$ for $|x|=R$. The function $g$ is defined as follows:
\begin{equation}\label{g-s}
 g=\tau \eta_R+(1-\eta_R)v(\cdot,0),
\end{equation}
where $\eta_R$ is a smooth function depending only on $r=|x|$ such
that $\eta_R\equiv 1$ in $B_{R-1}$ and $\eta_R\equiv 0$ outside $B_R$.
Then we define $\overline{w}(x,\lambda)$ as the unique solution of
the Dirichlet problem
\begin{equation}\label{eqw-s}\begin{cases}
\text{div}(\lambda^{1-2s}\nabla \overline{w})=0 & \mbox{in} \;C_R\\
\overline{w}(x,0)=g(x)& \mbox{on}\;B_R \times \{\lambda=0\}\\
 \overline{w}(x,\lambda)=v(x,\lambda)& \mbox{on}\;\partial
 C_R\cap \{\lambda>0\}.\end{cases}\end{equation}
Since $v$ is a global minimizer of $\mathcal{E}_{s,C_R}$ and
$\overline{w}=v$ on $\partial C_R\times \{\lambda>0\}$, then
\begin{eqnarray}&&d_s\int_{C_R}\frac{1}{2}\lambda^{1-2s}|\nabla v|^2dx d\lambda
+\int_{B_R}\{G(u(x,0))-c_u\}dx \nonumber \\
&& \hspace{2em}\leq d_s\int_{C_R}\frac{1}{2}\lambda^{1-2s}|\nabla
\overline{w}|^2dx d\lambda
+\int_{B_R}\{G(\overline{w}(x,0))-c_u\}dx.\nonumber
\end{eqnarray} We prove now that
$$d_s\int_{C_R}\frac{1}{2}\lambda^{1-2s}|\nabla \overline{w}|^2dx d\lambda +\int_{B_R}\{G(\overline{w}(x,0))-c_u\}dx \leq C R^{n-2s}\int_{1/R}^1\rho^{-2s}d\rho.$$

First of all, observe that the potential energy is bounded by
$CR^{n-1}$. Indeed, by definition $\overline{w}(x,0)=\tau$ in
$B_{R-1}$, and therefore

\begin{eqnarray}&&\int_{B_R} \{G(\overline{w}(x,0))-c_u\}dx=\int_{B_R\setminus
B_{R-1}}\{G(\overline{w}(x,0))-c_u\}dx\nonumber \\
&& \hspace{2em}\leq C |B_R\setminus B_{R-1}|\leq
CR^{n-1}.\end{eqnarray}

Thus, we need to bound the Dirichlet energy. First of all,
rescaling, we set
$$\overline{w}_1(x,\lambda)=\overline{w}(R{x},R{\lambda}),$$ for
 $({x},{\lambda})\in C_1=B_1\times(0,1)$. Moreover, if we set
$\varepsilon=1/R$ then
$$\overline{w}_1({x},0)=\begin{cases}
\tau &\mbox{for}\;\;|{x}|<1-\varepsilon\\
v(R{x},0)&\mbox{for}\;\;|{x}|=1.\end{cases}$$ We
observe that
$$d_s\int_{C_R}\lambda^{1-2s}|\nabla \overline{w}|^2dx d\lambda=d_sR^{n-2s}\int_{C_1}\lambda^{1-2s}|\nabla
\overline{w}_1|^2dx d\lambda.$$ Thus, it is enough to prove that
\begin{equation}\label{est-epsilon-s}d_s\int_{C_1}\lambda^{1-2s}|\nabla
\overline{w}_1|^2dx d\lambda \leq C\int_{\varepsilon}^1\rho^{-2s}d\rho.
\end{equation}
Applying Theorem \ref{extension-s} (and Remark \ref{s-ext}) with $\Omega=C_1$ and $M=B_1\times \{0\}$, we have that
\begin{eqnarray*}
&&d_s\int_{C_1}\lambda^{1-2s}|\nabla {\overline{w}_1}|^2dx d\lambda \\
&&\hspace{2em} \leq C
||\overline{w}_1||^2_{L^2(\partial C_1)} + C
\int\int_{B_{\rm frac}}\frac{|\overline{w}_1(z)-\overline{w}_1(\overline{z})|^2}{|z-\overline{z}|^{n+2s}}d\mathcal H^n(z)d\mathcal H^n({\overline{z}})\\
&&\hspace{2em}+C \int\int_{B_{\rm weig}}
d_M(z)^{1-2s}\frac{|\overline{w}_1(z)-\overline{w}_1(\overline{z})|^2}{|z-\overline{z}|^{n+1}}d\mathcal H^n(z)d\mathcal H^n(\overline{z}),\end{eqnarray*}
where $B_{\rm frac}$ and $B_{\rm weig}$ are defined as in \eqref{Bfrac} and
\eqref{Bweig} with $A=\partial C_1$ and $M= B_1 \times
\{0\}$. 

To bound the two double integrals above, we apply Theorem
\ref{key-s} to ${\overline{w}_1}_{|_{\partial C_1}}$ in $A=\partial C_1$,
taking $\Gamma=\partial B_1\times \{\lambda=0\}$. Since $|\overline{w}_1|\leq
C$, we only need to check the gradient bounds \eqref{grad-1A} and \eqref{grad-2M} in $\partial C_1$. In
the bottom boundary, $M=B_1 \times \{0\}$, this is simple. Indeed
$\overline{w}_1\equiv \tau$ in $B_{1-\varepsilon}$, and thus we need only to
control $|\nabla \overline{w}_1(x,0)|=\varepsilon^{-1}|\nabla g(Rx)|\leq
C\varepsilon^{-1}$ for $|x|>1-\varepsilon$, where $g$ is defined
in \eqref{g-s}. We have used estimate \eqref{grad-1} on $v$. Here ${d}_\Gamma(x)<\varepsilon$,
and one can deduce that \eqref{grad-1A}, \eqref{grad-2M} hold here.

Next, to verify \eqref{grad-1A} and \eqref{grad-2M} in $\partial C_1 \cap
\{\lambda>0\}$ we use that $\overline w=v$ here and we know that
$v$ satisfies \eqref{grad-1}, \eqref{grad-2}, and \eqref{grad-3}.
Thus the tangential derivatives of $\overline{w}_1$ in $\partial C_1 \cap
\{{\lambda}>0\}$ satisfy
\begin{equation}\label{e1}
|\nabla_{{x}} \overline{w}_1({x},{\lambda})|\leq
c_sR=\frac{c_s}{\varepsilon }\:\:\mbox{for
}\;\;(x,\lambda)\in
\partial C_1\cap\{\lambda >0\}\end{equation}
\begin{equation}\label{e2}
|\nabla \overline{w}_1({x},{\lambda})|\leq
\frac{c_sR}{R{\lambda}}=\frac{c_s}{{\lambda}}\:\:
\mbox{for }\;\;(x,\lambda)\in
\partial C_1\cap\{\lambda >0\}
\end{equation} and
\begin{equation}\label{e3} |{\lambda}^{1-2s}\partial_{{\lambda}}
\overline{w}_1(x,\lambda)|\leq \frac{c_sR}{R^{1-2s}}=\frac{c_s}{\varepsilon^{2s}} \:\: \mbox{for
}\;\;(x,\lambda)\in
\partial C_1\cap\{\lambda >0\}.\end{equation}

Estimate \eqref{e2} is used on the top boundary $B_1\times \{\lambda=1\}$ to verify in this set \eqref{grad-1A} and \eqref{grad-2M}.
Note that here $d_{\Gamma}((x,\lambda))$ is comparable to $\lambda=1$ up to multiplicative constants. 
\eqref{e1}, \eqref{e2}, and \eqref{e3} also lead to \eqref{grad-1A} and \eqref{grad-2M} on the lateral boundary $\partial B_1\times (0,1)$, where $d_{\Gamma}((x,\lambda))=\lambda$.
Hence,
${\overline{w}_1}_{|_{\partial C_1}}$
satisfies the hypothesis of Theorem \ref{key-s}. We conclude that the estimate for the Dirichlet energy
(\ref{est-epsilon-s}) holds.
\end{proof}

\section{Energy estimate for monotone solutions in $\re^3$}
In section 5 of \cite{C-Cinti1}, we proved two technical lemmas which
led to the energy estimate for monotone solutions (without limit assumption) and $s=1/2$ in
dimension $n=3$. Here we give analogue results but for every
fractional power $0<s<1$ of the Laplacian.

The first lemma concerns the stability property of the limit
functions
$$\underline{v}(x_1,x_2,\lambda):=\lim_{x_3 \rightarrow
-\infty}v(x,\lambda)\;\;\mbox{and}\;\;\overline{v}(x_1,x_2,\lambda):=\lim_{x_3 \rightarrow +\infty}v(x,\lambda)
,$$ and some properties of the potential
$G$ in relation with these functions. The second proposition establishes that monotone solutions
are global minimizers among a suitable class of functions which in turn allows us to apply a
comparison argument to obtain energy estimates.

\begin{lem}\label{lemma-s}
Let $f$ be a $C^{1,\gamma}$  function, for some $\gamma>\max\{0,1-2s\}$, and
$u$ be a bounded solution of equation (\ref{eq1-s}) in $\re^3$, such
that $u_{x_3}>0$. Let $v$ be the $s$-extension of $u$ in $\re^{4}_+$.

Set
$$\underline{v}(x_1,x_2,\lambda):=\lim_{x_3 \rightarrow
-\infty}v(x,\lambda)\;\;\mbox{and}\;\; \overline{v}(x_1,x_2,\lambda):=\lim_{x_3 \rightarrow +\infty}v(x,\lambda)
.$$

Then, $\underline{v}$ and $\overline{v}$ are solutions of \eqref{eq2-s} in $\re^3_+$, and each of them is either constant or
it depends only on $\lambda$ and on one Euclidian variable in the
$(x_1,x_2)-$plane. As a consequence, each $\underline{u}=\underline v(\cdot,0)$ and $\overline{u}=\overline
v(\cdot,0)$ is either constant or 1-D.

Moreover, set $m=\inf
\underline{u}\leq \widetilde m =\sup
\underline{u}$ and $\widetilde{M}=\inf \overline{u}\leq M =\sup \overline{u}$.

Then, $G >
G(\widetilde{m})=G(m)$ in $(m,\widetilde{m})$,
$G'(\widetilde{m})=G'(m)=0$ and $G > G(\widetilde{M})=G(M)$ in $(\widetilde{M},M)$,
$G'(\widetilde{M})=G'(M)=0.$
\end{lem}

\begin{proof} The proof is
the same as in the case of the half-Laplacian
(see \cite{C-Cinti1}). We do not supply all details and just recall the two main steps:
\begin{enumerate}
\item show that the functions $\underline{v}$ and $\overline{v}$
are stable solutions of problem \eqref{eq2-s} in $\re^3_+$ and thus their trace in $\re^2$ is 1-D by the one-dimensional symmetry result of the first author and Sire, Theorem 2.12 of \cite{C-Si2}; \item
apply Theorem 2.2 (i) of the first author and Sire \cite{C-Si1}, which gives necessary conditions on the
nonlinearities $f$ for which there exists an increasing solution to
 \eqref{eq2-s} in dimension $n=1$. This leads to the conditions on $G$ stated at the end of the lemma.
\end{enumerate}
\end{proof}

\begin{prop}\label{monot-min-s}
Let $f$ be any $C^{1,\gamma}$ nonlinearity, for some $\gamma>\max\{0,1-2s\}$. Let $u$ be a bounded solution of (\ref{eq1-s}) in $\re^n$
such that
$u_{x_n}>0$, and let $v$ be the $s$-extension of $u$ in $\re_+^{n+1}$.

Then, \begin{eqnarray*}d_s\int_{C_R}\frac{1}{2}\lambda^a|\nabla
v(x,\lambda)|^2
dx d\lambda&+&\int_{B_R}G(v(x,0))dx \\
&&\leq d_s\int_{C_R}\frac{1}{2}\lambda^a|\nabla w(x,\lambda)|^2 dx
d\lambda+\int_{B_R}G(w(x,0))dx,\end{eqnarray*} for every $H^1(C_R,\lambda^a)$ function $w$ such that $w=v$ on $\partial^+C_R=\partial C_R \cap \{\lambda>0\}$ and $\underline{v}\leq
w\leq \overline{v}$ in $C_R$, where $\underline{v}$ and
$\overline{v}$ are defined by
$$\underline{v}(x',\lambda):=\lim_{x_n \rightarrow
-\infty}v(x',x_n,\lambda)\;\;\mbox{and}\;\; \overline{v}(x',\lambda):=\lim_{x_n \rightarrow +\infty}v(x',x_n,\lambda)
.$$
\end{prop}

\begin{proof}
As in the case of the half-Laplacian, this property of local
minimality of monotone solutions $w$ such that $\underline{v}\leq
w\leq \overline{v}$ follows from the following two results:

\begin{enumerate}
\item [i)] Uniqueness of the solution to the problem
\begin{eqnarray}\label{eq2-ball-s}
\begin{cases}
\text{div}(\lambda^a\nabla w)=0& \text{in}\; C_R,\\
w=v& \text{on}\; \partial^+C_R=\partial C_R\cap \{\lambda>0\},\\
- d_s\lambda^a\partial_{ \lambda}w=f(w)& \text{on}\; B_R,\\
\underline{v}\leq w\leq \overline{v}&\text{in}\; C_R.
\end{cases}
\end{eqnarray} Thus, the solution must be $w\equiv v$.
This is the analogue of Lemma 3.1 of \cite{C-SM} for $s=1/2$, and below we comment on its proof. In this fractional case, it is stated in Lemma 5.1 of \cite{C-Si2}.
  \item [ii)] Existence
of an absolute minimizer for \eqref{eq2-ball-s}, that is, for $\mathcal{E}_{s,C_R}$ in the set
$$C_v=\{w\in H^1(C_R,\lambda^a)\,|\,w \equiv
v\:\mbox{on}\:\partial^+C_R,\:\underline{v}\leq w\leq
\overline{v}\;\mbox{ in}\; C_R\}.$$
\end{enumerate}

The statement of the proposition follows from the fact that by i) and ii), the monotone solution $v$, by uniqueness, must agree with
the absolute minimizer in $C_R$.

To prove points i) and ii), we proceed exactly as in \cite{C-Si2}.
For this, it is important that $\underline v$ and $\overline v$ are respectively, a strict subsolution and a strict supersolution of the Dirichlet-Neumann mixed problem \eqref{eq2-ball-s}.
We make a short
comment about these proofs.

\begin{enumerate}
\item [i)] The proof of uniqueness is based on sliding the
function $v(x,\lambda)$ in the direction $x_n$. We set
$$v^t(x_1,...,x_n,\lambda)=v(x_1,...,x_n+t,\lambda)\quad \mbox{for every}\:(x,\lambda)\in \overline{C}_R.$$
Since $v^t\rightarrow \overline{v}$ as $t\rightarrow +\infty$
uniformly in $\overline{C}_R$ and $\underline{v}< w<
\overline{v}$ (here we use that $w$ solves \eqref{eq2-ball-s} and that $\underline v$ is a subsolution to guarantee $\underline v < w$), then $w<v^t$ in $\overline{C}_R$, for $t$ large
enough. We want to prove that $w<v^t$ in $\overline{C}_R$ for
every $t>0$. Suppose that $s>0$ is the infimum of those $t>0$ such
that $w<v^t$ in $\overline{C}_R$. Then by
applying maximum principle and Hopf's lemma we get a contradiction, since one would have $w\leq v^s$ in $\overline C_R$ and $w=v^s$ at some point in $\overline C_R \setminus \partial^+ C_R$.

\item [ii)] To prove the existence of an absolute minimizer for
$\mathcal{E}_{C_R}$ in the convex set $C_v$, we proceed exactly as
in Lemma 4.1 of \cite{C-Si2}, substituting $-1$ and
$+1$ by the subsolutions and supersolution $\underline{v}$ and $\overline{v}$, respectively.
\end{enumerate}
\end{proof}

We give now the proof of the energy estimate in dimension 3 for
monotone solutions without the limit assumptions.

\begin{proof}[Proof of Theorem \ref{energy-dim3-s}]
We follow the proof of Theorem 5.2 of \cite{AAC}. We need to prove that the comparison function $\overline w$, used in the proof of Theorem \ref{energy-est-s}, satisfies $\underline v\leq \overline w\leq \overline v$. Then we can apply Proposition
\ref{monot-min-s} to make the comparison argument with the function $\overline w$ (as for global minimizers).
We recall that $\overline w$ is the solution of
\begin{equation}\label{pb-w-s}\begin{cases}
\mbox{div}(\lambda^{1-2s}\nabla \overline w)=0  & \mbox{in} \;C_R\\
\overline{w}(x,0)=g(x)& \mbox{on}\;B_R\times \{\lambda=0\}\\
\overline{w}(x,\lambda)=v(x,\lambda)& \mbox{on}\;\partial
C_R\cap \{\lambda >0\},\end{cases}\end{equation}
where $g=\tau\eta_R+(1-\eta_R)v(\cdot,0)$ as in \eqref{g-s}. Recall that $\tau$ is such that $G(\tau)=c_u=\min\{G(s):\inf_{\re^n}u\leq s\leq \sup_{\re^n}u\}$. Thus, if we prove that we van take $\tau$ such that $\sup \underline v\leq \tau \leq \inf \overline v$, then $\underline v\leq g\leq \overline v$ and hence $\underline v$ and $\overline v$ are respectively, subsolution and supersolutions of \eqref{pb-w-s}. It follows that $\underline v\leq \overline w\leq \overline v$, as desired.

To show that $\sup \underline v\leq \tau \leq \inf \overline v$, let
$m=\inf u=\inf \underline{u}$ and $M=\sup u=\sup \overline{u}$,
where $\underline{u}$ and $\overline{u}$ are defined in Lemma
\ref{lemma-s}. Set $\widetilde{m}=\sup \underline{u}$ and
$\widetilde{M}=\inf \overline{u}$; obviously $\widetilde{m}$ and
$\widetilde{M}$ belong to $[m,M]$. By Lemma \ref{lemma-s},
$\underline{u}$ and $\overline{u}$ are either constant or monotone
1-D solutions, moreover
\begin{equation}\label{G1-s}
G>G(m)=G(\widetilde{m})\;\;\mbox{ in}\;
(m,\widetilde{m})\end{equation} in case $m<\widetilde{m}$ (i.e.
$\underline u$ not constant), and
\begin{equation}\label{G2-s}G>G(M)=G(\widetilde{M})\;\;\mbox{ in}\;(\widetilde{M},M)\end{equation}
 in case $\widetilde{M}<M$ (i.e.
$\overline u$ not constant).

In all four possible cases (that is, each
$\underline{u}$ and $\overline{u}$ is constant or
one-dimensional), we deduce from (\ref{G1-s}) and (\ref{G2-s}) that
$\widetilde{m}\leq \widetilde{M}$ and that there exists $\tau\in
[\widetilde{m},\widetilde{M}]=[\sup \underline v=\sup \underline u, \inf \overline v=\inf \overline u]$ such that $G(\tau)=c_u$ (recall that
$c_u$ is the infimum of $G$ in the range of $u$), as desired. 
\end{proof}

\section{1-D symmetry in $\re^3$}

To prove Theorem \ref{degiorgi-s} we follow the argument, used by
Ambrosio and the first author \cite{AC} in their proof of the conjecture of De Giorgi
in dimension $n=3$. It relies on a Liouville type theorem. We
recall an adapted version of this result for the fractional case,
given by the first author and Sire (Theorem 4.10 in \cite{C-Si1}).

\begin{teo}\label{liouville-s}(\cite{C-Si1})
Let $a\in (-1,1)$, $\varphi \in
L^{\infty}_{loc}(\overline{\re_+^{n+1}})$ be a positive function and suppose
that $\sigma \in H^1_{loc}(\overline{\re_+^{n+1}}, \lambda^a)$ is a solution of
\begin{equation}\label{liouveq-s}
\begin{cases}
-\sigma{\rm div}(\lambda^a \varphi^2 \nabla\sigma) \leq 0
&\quad \hbox{in } \re^{n+1}_+\\
-\sigma\lambda^a\frac{\partial \sigma}{\partial \lambda } \leq 0 &\quad
\hbox{on }
\partial\re^{n+1}_+
\end{cases}
\end{equation}
in the weak sense.
Moreover assume that for every $R >1$,
\begin{equation}\label{intR2-s}
\int_{C_R} \lambda^a (\varphi \sigma)^2 dx d\lambda \leq C R^2,
\end{equation}
for some constant $C$ independent of $R$.

Then, $\sigma$ is constant.
\end{teo}

We can now give the proof of our one-dimensional symmetry result.

\begin{proof}[Proof of Theorem \ref{degiorgi-s}]
We follow the
proof of Theorem 1.4 in \cite{C-Cinti1}, where the result was established for $s=1/2$.
Hence, here we may assume $s>1/2$.

First of all observe that
both global minimizers and monotone solutions are stable. Thus, in
both cases (see Lemma 6.1 of \cite{C-Si2}), there exists a H\"older continuous function $\varphi$ in $\overline{\re^{4}_+}$
 such that
$\varphi >0$ in $\overline{\re^{4}_+}$, $\varphi \in H^1_{\rm{loc}}(\overline{\re^{4}_+},\lambda^a)$, and
\begin{equation*}
\begin{cases}
\mbox{ div}(\lambda^a \nabla \varphi) = 0
&\quad \hbox{in } \re^{4}_+\\
-d_s\lambda^a \partial _\lambda \varphi =f'(v)\varphi &\quad
\hbox{on }
\partial\re^{4}_+.
\end{cases}
\end{equation*}
Note that, if $u$ is a monotone solution in the direction $x_3$, say,
then we can choose $\varphi=v_{x_3}$, where $v$ is the
$s$-extension of $u$ in the half space. For $i=1,2,3$ fixed,
consider the function:
$$\sigma_i=\frac{v_{x_i}}{\varphi}.$$

We prove that $\sigma_i$ is
constant in $\re^{4}_+$ using the Liouville type Theorem
\ref{liouville-s} and our energy estimate.

We have that
$$\mbox{div}(\lambda^a\varphi^2\nabla \sigma_i)=0\quad
\mbox{in}\;\;\re^{4}_+.$$ Moreover $-\lambda^a \partial_\lambda \sigma_i$ is zero on $\partial \re^{4}_+$.
Indeed
$$\lambda^a\varphi^2\partial_\lambda\sigma_i=\lambda^a\varphi v_{\lambda
x_i}-\lambda^av_{x_i}\varphi_\lambda=0$$ since both $v_{x_i}$ and
$\varphi$ satisfy the same boundary condition
$$-d_s\lambda^a\partial_{\lambda} v_{x_i} -f'(v)v_{x_1}=0,\;\;\;
-d_s\lambda^a \partial_\lambda \varphi-f'(v)\varphi=0.$$

By Theorems \ref{energy-est-s} and \ref{energy-dim3-s}, $v$ satisfies the energy estimate 
(\ref{energy-s3}). Since $G(u)-c_u\geq 0$ in $\re^3$, we deduce
$$\int_{C_R}\lambda^{1-2s}(\varphi\sigma_i)^2\leq \int_{C_R}\lambda^{1-2s}|\nabla v|^2 \leq CR^2
,\quad \mbox{for every} \;R>2\;\;\mbox{and} \;\;1/2<s<1.$$ Thus, using
Theorem \ref{liouville-s}, we deduce that $\sigma_i$ is constant for
every $i=1,2,3$. We get
$$v_{x_i}=c_i \varphi \quad \mbox{for some constant}\;\;
c_i,\;\;\;\mbox{with}\:\:i=1,2,3.$$ We conclude the proof
observing that if $c_1=c_2=c_3=0$ then $v$ is constant. Otherwise
we have
$$
c_iv_{x_j}-c_jv_{x_i}=0\quad \mbox{for every}\:i\neq j,$$ and we
deduce that $v$ depends only on $\lambda$ and on the variable
parallel to the
vector $(c_1,c_2,c_3)$. Thus $u(x)=v(x,0)$ is 1-D.
\end{proof}

\end{document}